\documentclass[preprint,3p]{elsarticle}
\usepackage{graphicx, algorithmic, algorithm2e, psfrag, amssymb, amsthm, lineno, amsfonts, amsmath, pstricks-add, bm, float, comment, amsthm,color}
\usepackage{bbm,mathtools,mathdefs}
\RestyleAlgo{boxruled}
\usepackage[caption=false]{subfig}
\newtheorem{thm}{Theorem}
\newtheorem{lem}{Lemma}
\newcommand{\mc}[1]{\mathcal{#1}}
\newlength{\leftstackrelawd}
\newlength{\leftstackrelbwd}
\def\leftstackrel#1#2{\settowidth{\leftstackrelawd}%
{${{}^{#1}}$}\settowidth{\leftstackrelbwd}{$#2$}%
\addtolength{\leftstackrelawd}{-\leftstackrelbwd}%
\leavevmode\ifthenelse{\lengthtest{\leftstackrelawd>0pt}}%
{\kern-.5\leftstackrelawd}{}\mathrel{\mathop{#2}\limits^{#1}}}

\newtheorem{remark}{Remark}
\journal{Elsevier}
\begin{document}
\begin{frontmatter}
\title{Parametric/Stochastic Model Reduction: Low-Rank Representation, Non-Intrusive Bi-Fidelity Approximation, and Convergence Analysis}
\author[label1]{Jerrad Hampton}
\author[label2]{Hillary Fairbanks}
\author[label3]{Akil Narayan}
\author[label1]{Alireza Doostan\corref{cor1}}
\ead{alireza.doostan@colorado.edu}
\cortext[cor1]{Corresponding Author: Alireza Doostan}

\address[label1]{Smead Aerospace Engineering Sciences Department, University of Colorado, Boulder, CO 80309, USA}
\address[label2]{Department of Applied Mathematics, University of Colorado, Boulder, CO 80309, USA}
\address[label3]{Department of Mathematics, and Scientific Computing and Imaging (SCI) Institute, University of Utah, Salt Lake City, UT 84112, USA}

\begin{abstract}

For practical model-based demands, such as design space exploration and uncertainty quantification (UQ), a high-fidelity model that produces accurate outputs often has high computational cost, while a low-fidelity model with less accurate outputs has low computational cost. It is often possible to construct a bi-fidelity model having accuracy comparable with the high-fidelity model and computational cost comparable with the low-fidelity model. This work presents the construction and analysis of a non-intrusive (i.e., sample-based) bi-fidelity model that relies on the low-rank structure of the map between model parameters/uncertain inputs and the solution of interest, if exists. Specifically, we derive a novel, pragmatic estimate for the error committed by this bi-fidelity model. We show that this error bound can be used to determine if a given pair of low- and high-fidelity models will lead to an accurate bi-fidelity approximation. The cost of this error bound is relatively small and depends on the solution rank. The value of this error estimate is demonstrated using two example problems in the context of UQ, involving linear and non-linear partial differential equations. 

\end{abstract}

\end{frontmatter}

\section{Introduction}
\label{sec:intro}

In many practical contexts, an ideal computational method accurately recovers the physical phenomena that it is tasked to model, and does so in a computationally efficient manner that allows repeated calculations of quantities of interest (QoI's) for a large number of different input parameters. When this is possible, the QoI's may be understood across a range of input parameters, or similarly for certain distributions of input parameters. The latter case is connected with the field of parametric uncertainty quantification (UQ), and is fundamental for understanding the propagation of uncertainty through complex models often via the construction of the map between the inputs and QoI's. A growing area of interest in model-based simulations is a fast construction of such a mapping, which also forms our motivation here. For physical systems that are expensive to simulate, such a task may not be straightforward and may provide a significant challenge as it often requires a tradeoff between computational cost and accuracy.

In practice, the model which accurately represents the underlying physics -- the {\it high-fidelity} model -- may be expensive, and many repeated evaluations are likely computationally infeasible. Alternatively, one may consider the use of a {\it low-fidelity} model: This is a model with lower accuracy which is cheaper to evaluate -- relative to the high-fidelity model -- but not necessarily as accurate. The topic of this paper is a {\it bi-fidelity} approach that leverages both models and inherits the strengths of each.

In this work, we assume the existence of a low- and a high-fidelity model that each identify the same QoI in a discrete sense, i.e., as a vector, possibly with different sizes and levels of accuracy. For example, a fine spatial discretizaton of a heat equation, while costly, may lead to accurate temperature predictions, but a coarse spatial discretization gives inaccurate estimates that are generated more quickly. The notion of accuracy here is contextual, both with respect to the metric and the desired quantitative tolerance. Mathematically, we take the high-fidelity model as one that gives an approximation to the {\it truth} to within a desired accuracy and ideally with smallest possible computational cost. It is critical that we assume such a high-fidelity model is still expensive to evaluate, so that methods that would require many simulations from this model are impractical. We also assume that it is (much) cheaper to sample from the low-fidelity model so that compiling a larger number of samples from the low-fidelity model is possible. The goal of a (non-intrusive) bi-fidelity, or more generally multi-fidelity, model is to use relatively few samples from the high-fidelity model and a larger number of samples from the low-fidelity model to generate approximations to the QoI -- at arbitrary input samples --  with accuracies comparable to that of the high-fidelity model.

Early work in multi-fidelity methods for parametric problems has a decidedly statistical flavor. Such work emerged in the field of geosciences with the application of Gaussian process regression -- also known as kriging or co-kriging in the multivariate case -- in a multi-fidelity setting. In \cite{Kennedy00}, an autoregressive scheme is considered to generate a Gaussian process approximation of the output of the most expensive model from nested observations of multiple, less expensive models. In this autoregressive scheme, a given model is approximated by the lower-fidelity model through a multiplicative constant and an additive Gaussian process correction term.  In a related work, \cite{Qian06} creates a Gaussian process surrogate to the low-fidelity model, which is subsequently adjusted via samples of the high-fidelity model. The adjustment is done by training a similar autoregressive model as in \cite{Kennedy00}, consisting in a linear -- in input parameters -- mapping of the low-fidelity surrogate and an additive discrepancy term modeled as a Gaussian process. The resulting high-fidelity surrogate is also a Gaussian process. Extensions of multi-fidelity Gaussian process modeling for applications in optimization are explored in \cite{Forrester07, Laurenceau08}, and recursive methods for improved accuracy or computational complexity are developed in \cite{Kleiber13,LeGratiet14,LeGratiet15,Perdikaris15,Perdikaris16,Parussini17}, among other work. 
 
There has also been significant work over the last several years on reformulating other prominent UQ methods in a multi-fidelity (or multi-level) framework. These methods generally rely on a relatively large number of low-fidelity samples along with a (much) smaller number of high-fidelity samples to build an additive and/or multiplicative correction of the low-fidelity model. The adjusted low-fidelity model will then serve as a multi-fidelity approximation to the high-fidelity model. In particular, in multi-fidelity polynomial chaos expansions (PCEs) \cite{Eldred09, Ng12, Palar15, Padron16}, the high-fidelity solution is estimated by first determining the PCE of the low-fidelity solution, and subsequently the PCE of a correction term is generated from a small set of low-fidelity and high-fidelity samples.  In multi-level Monte Carlo (MC) \cite{Heinrich01,Giles08,Cliffe11}, a sequence of fine-to-coarse level (or grid) models are considered and the expectation of the QoI is estimated via a telescoping summation. These methods apply MC to the coarsest level QoI, and then add the MC estimates of the QoI differences, a.k.a. corrections, over each pair of consecutive grids. Assuming the corrections over finer grids become smaller, fewer realizations of the finer (high-fidelity) grid QoI are needed to compute the correction expectations, thus resulting in a reduced overall cost.

With an aim of constructing a reduced polynomial representation of the high-fidelity solution, the work in \cite{Doostan07,Ghanem07a} builds a small size polynomial basis -- in the random inputs -- identified by a Karhunen-Lo\`eve expansion of the low-fidelity solution. The high-fidelity solution is then approximated via Galerkin projection of the governing equations on the span of the reduced polynomial basis -- as opposed to the standard PCE basis -- resulting in a much smaller size system of equations to solve. The key assumption in this bi-fidelity construction is that the largest eigenvalues of the covariance matrix of the discrete solution decay quickly; that is, the covariance matrix and, therefore, the solution are low-rank. A non-intrusive implementation of this work is found in \cite{Kumar16}.

Of most relevance to the approach of this study, \cite{Narayan14b, Zhu14} considers the low-rank approximation of the discretized QoI $\bm v \in \R^M$,
\begin{equation}
\label{eq:low_rank}
\bm{v}(\bm \mu)\approx \sum_{\ell=1}^{r}\bm{v}(\bm \mu^{(\ell)})c_\ell(\bm\mu).
\end{equation}
The approximation above is of the same type constructed in so-called {\it reduced basis} representations in the (parametric) model reduction literature~\cite{Almroth78,Noor1980,Peterson89, Maday02,Rozza2008}. However, the procedure considered in \cite{Narayan14b, Zhu14} and here is disparate from the algorithmic strategies in the reduced basis community. In (\ref{eq:low_rank}), $\bm v$ at an arbitrary parametric input $\bm\mu\in\mathbb{R}^{d}$ is represented in a basis consisting of $r\ll M$ realizations of $\bm v$ -- corresponding to selected input samples $\bm \mu^{(\ell)}$ -- via the coefficients $c_\ell(\bm \mu)$. The work in \cite{Narayan14b, Zhu14} relies on a greedy procedure applied to low-fidelity realizations of $\bm v$ to identify the set of $\bm \mu^{(\ell)}$. The coefficients $c_\ell(\bs{\mu})$ are computed via an interpolatory condition on the low-fidelity model. %$\bm v(\bm\mu)$, at an arbitrary $\bm\mu$, in the resulting low-fidelity, reduced basis to compute the coefficients $c_\ell(\bm\mu)$. 
The same interpolation rule, i.e., the same coefficients $c_\ell(\bm\mu)$, is then used with the high-fidelity counterpart of the reduced basis snapshots and this results in the bi-fidelity construction (\ref{eq:low_rank}). In this manner, the high-fidelity model is evaluated only at a (small) number of input samples that is equal to the size $r$ of the reduced basis. Unlike in standard reduced basis techniques this approach does not perform Galerkin (or Pertov-Galerkin) projections of the high-fidelity equations, thus allowing the use of legacy codes in a black-box fashion.

With a different algorithm to construct the reduced basis and the interpolation rule, the work in \cite{Doostan16} and \cite{Skinner17} examines a similar bi-fidelity approximation on problems involving heat transfer in a ribbed channel and prediction of pressure coefficient over a family of NACA airfoils, respectively. The strategy in \cite{Doostan16} was used in \cite{Fairbanks2016} to build control variates within a multi-level MC framework. 

\subsection{Contribution of This Work}

While in practice the above multi-fidelity methods have shown promising results, there is a lack of convenient tools to verify the convergence of the multi-fidelity solution and understand the role of key factors affecting the convergence. These tools must also enable a practitioner to determine, {\it a priori} and with relatively small cost, whether or not a given low-fidelity model will result in accurate multi-fidelity approximations. To this end, our primary contribution is the derivation of an error bound that is appropriate for methods as in~\cite{Narayan14b, Zhu14,Doostan16,Fairbanks2016,Skinner17} for bi-fidelity approximation, with a particular emphasis on the role of the rank of the $\bm v$ ensemble in the approximation error. This bound is derived specifically for the bi-fidelity procedure in \cite{Doostan16,Fairbanks2016,Skinner17} and relies on quantities that are easily estimated with a number of high-fidelity samples that scales favorably with the rank of the $\bm v$ ensemble. As a result, the associated error estimate is efficient, easy to compute, and reliably conservative, making it desirable in practical contexts. This method is also presented in its full generality while applied to two examples that demonstrate the potential for practical utility of this error estimate. 

The remainder of this manuscript is organized into three main sections. Section~\ref{sec:detail} reviews the mathematical notation and procedure for the bi-fidelity approximation. Section~\ref{sec:theorem} presents our main result, which focuses on bounding the bi-fidelity approximation error. To show the accuracy of this bound, Section~\ref{sec:examples} provides two numerical experiments and illustrates the utility of the error estimate. Finally, Section \ref{sec:conc} gives a short summary of the conclusions of this study. In~\ref{sec:theory}, we provide the proof of the main error bound introduced in Section \ref{sec:theorem}. 

\section{Method Detail}
\label{sec:detail}

As in \cite{Narayan14b, Zhu14,Doostan16,Fairbanks2016,Skinner17}, we consider a bi-fidelity construction of a QoI admitting a low-rank representation of the form (\ref{eq:low_rank}). We presume here that these QoI's are functions defined over the spatial or temporal domain of the problem or can be computed from such functions. While the QoI's could contain the solution over the entire domain of the problem, e.g., temperature over the physical domain, this is often unnecessary and a reduction to some summary of the solution, e.g., temperature along a boundary, can explain the behavior of interest. In a broad sense, this low-rank heuristic implies the existence of a limited number of behaviors for $\bm v$ to exhibit, and that for any particular set of inputs $\bm\mu$, the corresponding $\bm v$ is nearly a linear reconstruction of $\bm v$ exhibiting these different traits. In this way, even $\bm v$ for possibly highly non-linear problems may be approximated, with some tolerable error, by a linear combination of relatively few basis vectors. This low-rank assumption is actually more than a heuristic: It can be codified via the mathematical notion of $n$-widths \cite{pinkus_n-widths_1985}, and QoI ensembles for differential equations can indeed exhibit this low-rank property \cite{bachmayr_kolmogorov_2017}.

This formulation raises two primary concerns: How do we identify the basis vectors, and how do we identify the appropriate linear combinations of these basis vectors? For computational feasibility, we require solutions to both these concerns that require neither a large number of high-fidelity realizations of $\bm v$, nor modification to the simulation codes.

To present our approach for addressing these concerns, and without loss of generality, we consider the approximation of a collection of $N$ high-fidelity realizations of the QoI, instead of individual realizations as in (\ref{eq:low_rank}). We denote the $k$th realization of the QoI obtained via the high- and low-fidelity models by $\bm v^{(k)}_H\in\mathbb{R}^M$ and $\bm v^{(k)}_L\in\mathbb{R}^m$, respectively. These vectors may be of different sizes, but depend on the same parametric input sample $\bm\mu^{(k)}$. For the examples presented in Section \ref{sec:examples}, we consider $\bs{\mu}$ as a random variable, and $\bm{v}_H^{(k)}$ and $\bm{v}_L^{(k)}$ are drawn in an MC fashion from the joint probability distribution on $\bm\mu$. When no probabilistic information on $\bm\mu$ is available, $\bm{v}_H^{(k)}$ and $\bm{v}_L^{(k)}$ may be generated from uniformly distributed samples of $\bm\mu$. 

Our analysis relies on arranging $\{\bm{v}^{(k)}_L\}$ and $\{\bm{v}^{(k)}_H\}$ into matrices, denoted $\bm{H} \in \R^{M \times N}$ and $\bm{L}\in \R^{m \times N}$, to refer to high-fidelity and low-fidelity data, respectively, where the corresponding columns of $\bm{H}$ and $\bm{L}$ are the QoI realizations at the same input samples. That is, $\bm{H}$ and $\bm{L}$ have the same number of columns, $N$, but generally have different numbers of rows $M$ and $m$, respectively. The low-rank assumption on the $\bm v_H$ ensemble implies that $\bm H$ admits a low-rank factorization. We assume the same for $\bm L$ but do not require the same rank for $\bm H$ and $\bm L$.  

\begin{remark}
For the interest of presentation and analysis we refer to the full high-fidelity matrix $\bm H$. However, in practice, we require having access to only $r\ll N$ ($r$ being the rank of $\bm v_H$ as in (\ref{eq:low_rank})) high-fidelity realizations to generate this bi-fidelity approximation; see Section \ref{sec:lifting}. To evaluate the bound on the resulting approximation error, we require a number of high-fidelity realizations that is slightly larger than $r$, as illustrated in the examples of Section \ref{sec:examples}. 
\end{remark}

\subsection{Low-Rank Factorization of Low-Fidelity Data}
\label{sec:rrqr}

The method used to identify a reduced basis and an interpolation rule for $\bm H$ (or $\bm v_H$) relies on identifying the same for the low-fidelity data $\bm L$. For this purpose, we form the rank $r$ {\it interpolative decomposition} of $\bm{L}$~ \cite{Cheng05,Martinsson11,Halko2011} using the rank-revealing QR factorization~\cite{GolubAndVanLoan}. The revealed rank of this factorization is our reduced rank $r$, and this provides the rank-$r$ reduced decomposition
\begin{align}
\label{eqn:QR_main}
\bm{L}\bm{P} &\approx  \bm{Q}_L\bm{R}_L,
\end{align}
for an $N\times N$ permutation matrix $\bm{P}$. Then $\bm{Q}_L$ has $r$ orthonormal columns, and $\bm{R}_L$ has $r$ rows and is upper triangular. A column partitioning of $\bm{R}_L$ into an $r\times r$ upper triangular matrix $\bm{R}_L^{(11)}$ and an $r \times (N-r)$ matrix $\bm{R}_L^{(12)}$, and replacing $\bm{R}_L^{(12)}$ with the approximation $\bm{R}_L^{(12)} \approx \bm{R}_L^{(11)} \bm Z$ for some $\bs{Z}$ leads to 
\begin{align}
\label{eqn:QR_partition}
\bm{L}\bm{P} &\approx  \bm{Q}_L\bm{R}_L^{(11)}\left[ \bm I\ |\ \bm{Z} \right],
\end{align}
where $\bm I$ is the $r\times r$ identity matrix. When $\bm{R}_L^{(11)}$ is ill-conditioned, \cite{Cheng05} suggests a solution $\bm{Z}$ with minimum $\Vert \bm Z\Vert_{F}$, where $\Vert \cdot\Vert_{F}$ denotes the Frobenius norm. %Given that $\bm P^{-1}= \bm P^T$, 
The rank $r$ factorization (\ref{eqn:QR_partition}) can be rewritten as 
\begin{align}
\bm{L} & \approx  \bm{L}(r)\left[ \bm I\ |\ \bm{Z} \right]\bm{P}^T; \nonumber\\
	   \label{eqn:QR_interpolative}
           & = \bm{L}(r)\bm{C}_L;\\
           \label{eq:lo_fi_est}
           & \coloneqq \hat{\bm{L}},
\end{align}
where $\bm{L}(r) = \bm{Q}_L\bm{R}_L^{(11)}$ -- referred to as the {\it column skeleton} of $\bm L$ \cite{Cheng05} -- contains the first $r$ columns of $\bm{L} \bm{P}$ and $\bm{C}_{L}\coloneqq[\ \bm{I}\ |\ \bm{Z}\ ]\bm{P}^{T}$. In words, (\ref{eqn:QR_interpolative}) identifies a set of $r$ columns of $\bm L$, i.e., a reduced basis for $\bm L$, along with a coefficient matrix $\bm C_L$ specifying the linear combination of the basis vectors to approximate all columns in $\bm{L}$. Hence the name {\it interpolative decomposition}. If we identify the entries of $\bs{C_L}$ with the coefficients $c_{\ell}(\bs{\mu})$ in \eqref{eq:low_rank}, then (\ref{eqn:QR_interpolative}) is the matrix formulation of (\ref{eq:low_rank}) for $N$ low-fidelity realizations $\{\bm v_L^{(k)}\}$. 

For the sake of completeness, we next report some properties of the factorization (\ref{eqn:QR_interpolative}) from \cite{Martinsson11}, which we will use later. For detailed analyses of the interpolative decomposition, the interested reader is referred to \cite{Cheng05,Martinsson11,Halko2011}.

\begin{lem} (Lemma 3.1 of \cite{Martinsson11}.) Let $\Vert \cdot \Vert$ denote the matrix induced $\ell_2$ norm. For any positive integer $r\le\min\{m,N\}$, 
\begin{enumerate}
\item $\Vert \bm{C}_L \Vert\le \sqrt{r(N-r)+1}$,
\item $\bm{L} = \hat{\bm{L}}$, when $r=m$ or $r=N$, and
\item $\Vert \bm{L} - \hat{\bm{L}} \Vert\le\sqrt{r(N-r)+1}\ \sigma_{r+1}$ when $r<\min\{m,N\}$, where $\sigma_{r+1}$ is the $(r + 1)$st greatest singular value of $\bm{L} $.
\end{enumerate} 
\label{lem:martinsson}
\end{lem}

\begin{remark}
In practice, the approximation rank $r$ is not known {\it a priori} and thus the pivoted Gram-Schmidt step involved in the QR factorization (\ref{eqn:QR_main}) is continued with larger ranks until $\Vert \bm{L}- \hat{\bm{L}}\Vert$ reaches a specified accuracy.
\end{remark}

\begin{remark}
Although not investigated here, other factorizations such as those in~\cite{Mahoney09, Elhamifar09, Dyer15,perry2017} are equally applicable to produce an interpolative decomposition of $\bm L$.
\end{remark}

\subsection{High-Fidelity Approximation via Basis Update and Low-Fidelity Interpolation} 
\label{sec:lifting}

In addition to a rank $r$ factorization of $\bm L$, the permutation matrix $\bs{P}$ in the interpolative decomposition (\ref{eqn:QR_interpolative}) identifies the set of $r$ input samples $\bm\mu^{(k)}$ corresponding to the basis vectors in $\bm{L}(r)$. This identification is one primary reason for employing an interpolative decomposition instead of the more widely-used and $\ell^2$-optimal singular value decomposition (SVD) of $\bm L$. As in \cite{Narayan14b, Zhu14, Doostan16,Fairbanks2016,Skinner17}, our approach to approximate $\bm{H}$ is to replace $\bm{L}(r)$ with $\bm{H}(r)$, a matrix of vectors from the high-fidelity model corresponding to the same input samples $\bm\mu^{(k)}$ and with the same arrangement as in $\bm{L}(r)$. Stated differently, we consider the high-fidelity counterpart of $\bm{L}(r)$ as the reduced basis for $\bm H$ (or $\bm v_H^{(k)}$). In this case we set the bi-fidelity, rank $r$ approximation of $\bm H$, denoted by $\hat{\bm{H}}$, to
\begin{align}
\label{eq:hi_fi_est}
\hat{\bm{H}} &\coloneqq \bm{H}(r)\bm{C}_L,
\end{align}
where $\bm{C}_L$ is the coefficient matrix computed from the low-fidelity approximation (\ref{eqn:QR_interpolative}). In other words, the interpolation rule learned from the low-fidelity model is applied to approximate the high-fidelity realizations.  In the case that we have a construction of this form, where the difference between $\hat{\bm{L}}$ and $\hat{\bm{H}}$ is limited to the changes only in the basis vectors associated with the low- and high- fidelity models, we say that $\hat{\bm{L}}$ is \textit{lifted} to the approximation $\hat{\bm{H}}$.  For this paper we say that $\hat{\bm{L}}$ and $\hat{\bm{H}}$ are \textit{corresponding estimates} (of $\bs{L}$ and $\bs{H}$, respectively). Our analysis presented next applies to any such corresponding estimates.
 
\begin{remark}[Utility of $\hat{\bm H}$ for UQ]
While outside the scope of the present study, the bi-fidelity realizations $\hat{\bm v}_H^{(k)}$, i.e., the columns of $\hat{\bm{H}}$, may be used in place of $\bm v_H^{(k)}$ in methods such as MC simulation or its variants, sparse grid stochastic collocation \cite{Xiu05a}, non-intrusive PCE \cite{Hosder06,Doostan11a,Constantine12a}, etc., to approximate the statistics of $\bm v_H$ or perform sensitivity analysis. The accuracy of such estimates, however, depends on the accuracy of the bi-fidelity estimates $\hat{\bm v}_H^{(k)}$.
\end{remark}

\subsection{Convergence Analysis and Its Practical Application}
\label{sec:theorem}

We are now prepared to state the main result that bounds the bi-fidelity error, $\|\hat{\bm{H}}-\bm{H}\|$, where $\|\cdot \|$ is matrix induced $\ell_2$ norm. However, we first present a high-level discussion on our analysis approach. At the core of our convergence results is the assumption that there exists a matrix $\bm{T}$ with {\it bounded} $\Vert \bm{T}\Vert$  such that for a matrix $\bm{E}$ with {\it small} $\Vert \bm{E}\Vert$, representing the error in the approximation,
\begin{align}
\label{eqn:core_assumption}
\bm{H} = \bm{T}\bm{L} + \bm{E}.
\end{align}
%We specify one particular such matrix $\bs{T}$ in Section \ref{sec:theory}. 
The matrix $\bs{T}$ is essentially a lifting operator from low-fidelity discretization space to high-fidelity discretization space. Our theory outlined in~\ref{sec:theory} computes this matrix explicitly, but does not require any knowledge about the discretizations or assumptions between the low- and high-fidelity models. We emphasize that $\bs{T}$ is only needed for theoretical analysis and our error bounds do not require its construction.

It is straightforward to verify that $\|\hat{\bm{H}}-\bm{H}\|$ can be bounded as  
\begin{align}
\nonumber\|\bm{H}-\hat{\bm{H}}\| &\le \|\bm{H} - \bm{TL}\| + \|\bm{TL}-\bm{T}\hat{\bm{L}}\|+ \|\bm{T}\hat{\bm{L}}-\hat{\bm{H}}\|;\\
 \nonumber&\le \|\bm{H} - \bm{TL}\| + \|\bm{T}\|\|\bm{L}-\hat{\bm{L}}\| + \|\bm{TL}-\bm{H}\|\|\bm{C}_L\|;\\
\label{eq:the_one_bound_pre}&\le \left(1 + \|\bm{C}_L\|\right)\|\bm{E}\| + \|\bm{T}\|\|\bm{L}-\hat{\bm{L}}\|.
\end{align}
When $\Vert \bm{T}\Vert$ is bounded and $\Vert \bm{E}\Vert$ is small, (\ref{eq:the_one_bound_pre}) suggests an accurate bi-fidelity estimate as $\|\bm{C}_L\|$ is bounded, and $\|\bm{L}-\hat{\bm{L}}\|$ is small given that $\bm{L}$ is low-rank; see Lemma \ref{lem:martinsson}. In general, there is no guarantee of the existence of a mapping $\bm T$ of $\bm L$ to $\bm H$ that gives an $\bm E$ with small norm. We therefore seek to establish a condition on $\bm L$  (in relation to $\bm H$) to ensure the existence of such $\bm{T}$ and  $\bm{E}$ matrices.

For a finite $\tau\ge 0$, define 
\begin{align}
\label{eqn:eps_tau_def}
\epsilon(\tau) &\coloneqq \lambda_{\max}(\bm{H}^T\bm{H}-\tau \bm{L}^T\bm{L}),
\end{align}
where $\lambda_{\max}(\cdot)$ denotes the largest eigenvalue of a matrix, and $\bm{H}^T\bm{H}$ and $\bm{L}^T\bm{L}$ are the Gramians of the high- and low-fidelity matrices, respectively. We describe in~\ref{sec:theory} that if $\epsilon(\tau)$ is small enough then matrices $\bs{T}$ and $\bs{E}$ with our desired properties can be constructed.%, when, for a given $\bm L$ and $\tau$, $\epsilon(\tau)$ is small enough, $\bm{T}$ and $\bm{E}$ with the above said properties can be constructed.  

For a fixed $\tau$, $\epsilon(\tau)$ has various interpretations. It is the smallest $\epsilon$ such that
\begin{align}
\label{eqn:eps_condition_rewrote}
\epsilon\bm{I} + \tau\bm{L}^T\bm{L} - \bm{H}^T\bm{H}
\end{align}
is a positive semi-definite matrix. Second, $\epsilon(\tau)$ has a geometrical interpretation that demonstrates its significance in this context of bi-fidelity approximation. %Let $\|\cdot\|$ denote the standard $\ell_2$ norm of a vector. 
The above equation implies that $\epsilon(\tau)$ is the smallest value such that for any $\bm{x}\in\mathbb{R}^N$,
\begin{align*}
\|\bm{Hx}\|^2&\le \tau\|\bm{Lx}\|^2 + \epsilon(\tau)\|\bm{x}\|^2.
\end{align*}
This guarantees that for every $\bm{x}$, $\sqrt{\tau}\bm{Lx}$ is further from the origin than $\bm{Hx}$, to within a margin governed by $\epsilon(\tau)$ and $\|\bm{x}\|$. This ensures that every realized $\bm H\bm x$ can be reached by rotating $\bm L\bm x$ through the application of a scaling related to $\tau$, and adding a small correction which is bounded based on $\epsilon(\tau)$ and $\|\bm x\|$. The relationship between these interpretations is further discussed in~\ref{sec:theory}, and we note this rotation, scaling, and additive correction is essentially utilized to prove the following theorem completing the bound in (\ref{eq:the_one_bound_pre}).

\begin{thm}
\label{thm:algorithm}
For any $\tau\ge 0$, let $\epsilon(\tau)$ be as in (\ref{eqn:eps_tau_def}). Let $\hat{\bm{H}}$ and $\hat{\bm{L}}$ be corresponding estimates of rank $r$ with coefficients $\bm{C}_L$, and let $\sigma_k$ denote the $k$th largest singular value of $\bm{L}$. Then,%, then for any $k$ satisfying $1 \leq k \leq \mathrm{rank}(\bs{L})$, %not exceeding the rank of $\bm{L}$,
\begin{subequations}
\label{eqn:main_bound}
\begin{small}
\begin{align}\label{eqn:main_bound_rho}
\|\bm{H}-\hat{\bm{H}}\| &\le \mathop{\min}\limits_{\tau,\, k\le\textnormal{rank}(\bm{L})}
%\left[\left(1 + \|\bm{C}_L\|\right)\sqrt{\tau\sigma^2_{k+1} + \epsilon(\tau)} + \|\bm{L}-\hat{\bm{L}}\|\sqrt{\tau + \epsilon(\tau)\sigma_k^{-2}}\right]. \\
\rho_k(\tau) \\
\label{eq:error_func}
\rho_k(\tau) &\coloneqq \left[\left(1 + \|\bm{C}_L\|\right)\sqrt{\tau\sigma^2_{k+1} + \epsilon(\tau)} + \|\bm{L}-\hat{\bm{L}}\|\sqrt{\tau + \epsilon(\tau)\sigma_k^{-2}}\right],
\end{align}
\end{small}
\end{subequations}
When $k = \mathrm{rank}(\bs{L})$, we set $\sigma_{k+1} = 0$.
\end{thm}

The proof of this theorem is given in~\ref{sec:theory}; here we provide insight and implications of its result. 

An actual evaluation of the bound requires estimating $\tau$ and $\epsilon(\tau)$, which can be achieved by minimizing the bound over a range of $k$ and $\tau$; see Algorithm \ref{alg:bound}. A \naive{} attempt to estimate $\epsilon(\tau)$ from \eqref{eqn:eps_tau_def} seems to require the entire matrix $\bs{H}$. However, we propose in Section \ref{sec:error_estimate} a procedure that uses only a relatively small number of high- and low-fidelity realizations. Notice that aside from $\epsilon(\tau)$, the bound (\ref{eqn:main_bound}) has no dependence on the high-fidelity model. Let
\begin{align}
\label{eq:first_term_bd}
B_1 = \left(1 + \|\bm{C}_L\|\right)\sqrt{\tau\sigma^2_{k+1} + \epsilon(\tau)};\\
\label{eq:second_term_bd}
B_2 = \|\bm{L}-\hat{\bm{L}}\|\sqrt{\tau + \epsilon(\tau)\sigma_k^{-2}}.
\end{align}
The terms (\ref{eq:first_term_bd}) and (\ref{eq:second_term_bd}) may be small: Observe that $B_1$ is small when $\tau\sigma^2_{k+1}\ll 1$ and the optimal $\epsilon(\tau)$ is small. Recall $\Vert \bm{C}_L \Vert\le \sqrt{r(N-r)+1}$ from Lemma \ref{lem:martinsson}. When, for a given $k$,  $\sigma_{k+1}$ is small, i.e., the singular values of $\bm L$ decay quickly, the magnitude of $B_1$ is governed by the optimal $\epsilon(\tau)$. To see that $B_2$ is small, note that if the low-rank reconstruction is accurate for the low-fidelity model, then $\Vert \bm{L} - \hat{\bm{L}} \Vert\le\sqrt{r(N-r)+1}\ \sigma_{r+1}$ is small following Lemma \ref{lem:martinsson}. Given this, $B_2$ is also small when $\sigma_{k+1}/\sigma_{k}\ll 1$, $\tau\sigma^2_{k+1}\ll 1$, and the optimal $\epsilon(\tau)$ is small. We note that the requirements $\sigma_{k+1}/\sigma_{k}\ll 1$ and $\tau\sigma^2_{k+1}\ll 1$ highlight the significance of the low-rank assumption on $\bm L$ in the success of this bi-fidelity construction. 

Using Lemma \ref{lem:martinsson}, the bound (\ref{eqn:main_bound}) can be simplified to
%
%\begin{align}
%\|\bm{H}-\hat{\bm{H}}\| \le \mathop{\min}\limits_{\tau,\, k\le\textnormal{rank}(\bm{L})} &\bigg[\bigg(1 + \sqrt{r(N-r)+1}\bigg)\sqrt{\tau\sigma^2_{k+1} + \epsilon(\tau)}\nonumber\\
%                                    &+ \sqrt{r(N-r)+1}\ \left(\frac{\sigma_{k+1}}{\sigma_{k}}\right)\sqrt{\tau\sigma_k^{2} + \epsilon(\tau)}\bigg]\nonumber
%\end{align}
%%
%or 
%%
\begin{align}
\|\bm{H}-\hat{\bm{H}}\| &\le \mathop{\min}\limits_{\tau,\, k\le\textnormal{rank}(\bm{L})}\left[\sqrt{r(N-r)+1}\ \left(1+\frac{\sigma_{k+1}}{\sigma_{k}}\right)\sqrt{\tau\sigma^2_{k} + \epsilon(\tau)}\right].\nonumber
\end{align}
In our numerical experiments, we use the sharper error estimate in (\ref{eqn:main_bound}). 

We now address the computation of an optimal $\epsilon(\tau)$, or, more precisely an optimal pair $(\tau,\epsilon(\tau))$. Specifically we consider an estimate $\hat{\epsilon}(\tau)$ with relatively few evaluations of the high-fidelity model. Such an estimate allows us to approximate the bound (\ref{eqn:main_bound}) on $\|\bm{H}-\hat{\bm{H}}\|$, which measures the quality of the bi-fidelity approximation.% and the suitability of the low-fidelity model at hand. 

\subsection{Estimating $\epsilon(\tau)$ Using Limited High-Fidelity Data}
\label{sec:error_estimate}

The definition of $\epsilon (\tau)$ in (\ref{eqn:eps_tau_def}) depends on full realizations of the high-fidelity model which is untenable in many practical situations and its direct computation would significantly detract from the utility of the bi-fidelity error estimate. Consider Gramian matrices and an estimate $\hat{\epsilon}$ defined, for a normalizing constant $c$, by
\begin{align}
\bm{G}_H &\coloneqq \bm{H}^T\bm{H};\\
\bm{G}_L &\coloneqq \bm{L}^T\bm{L};\\
\label{eq:eps_tau_hat_def}
\hat{\epsilon}(\tau) &\coloneqq c\lambda_{\max}(\hat{\bm{G}}_H - \tau \hat{\bm{G}}_L),
\end{align}
for estimates $\hat{\bm{G}}_H$ and $\hat{\bm{G}}_L$ of $\bm{G}_H$ and $\bm{G}_L$, respectively. Recall that the number of columns of $\bm{H}$ and $\bm{L}$ is $N$. Let $n$ denote a small number of columns that are sub-sampled from $\bm{H}$ and $\bm{L}$, and $\hat{\bm{G}}_H$ and $\hat{\bm{G}}_L$ are the Gramian matrices associated with these $n$ columns%, instead of the $N$ columns used to construct $\bm{G}_H$ and $\bm{G}_L$. 
%For general sub-sampling schemes we would expect the singular values of ${\bm{G}}_H$ and $\bm{G}_L$ to be approximately $N/n$ times those of $\hat{\bm{G}}_H$ and $\hat{\bm{G}}_L$, respectively. 
We set $c=N/n$ as the normalizing constant throughout the remainder of this work. With these Gramian estimates, we can construct $\hat{\epsilon}(\tau)$ for any $\tau$ using (\ref{eq:eps_tau_hat_def}). To approximate the bound in (\ref{eqn:main_bound}), we may replace $\epsilon (\tau)$ with $\hat{\epsilon}(\tau)$, and then identify the minimum achieved value over a range of $k$ and $\tau$ values, noting that the search over $k$ requires only additionally knowing the appropriate singular values of $\bm L$. Our numerical results in Section \ref{sec:examples} show empirically that values of $n$ slightly larger than $r$ are sufficient to estimate the optimal pair $(\tau, \epsilon(\tau))$. 

Algorithm \ref{alg:bound} summarizes our proposed approach to evaluate the bound (\ref{eqn:main_bound}) as an estimate for the bi-fidelity error $\|\bm{H}-\hat{\bm{H}}\|$. We emphasize that Algorithm \ref{alg:bound} uses $\hat{\epsilon}(\tau)$ from (\ref{eq:eps_tau_hat_def}).  

\vspace{.3cm}
\begin{algorithm}[H]
\caption{Algorithm for estimating the error bound (\ref{eqn:main_bound}).} 
\label{alg:bound}
\begin{algorithmic}[1]
\STATE Initialize a vector of values for $\tau\ge 0$.
\STATE Use $n$ high- and low-fidelity realizations to set the Gramian\\ matrices $\hat{\bm{G}}_L$ and $\hat{\bm{G}}_H$ as explained in Section \ref{sec:error_estimate}.  
\STATE For all values of $\tau$ evaluate the corresponding $\hat{\epsilon}(\tau)$ using (\ref{eq:eps_tau_hat_def}).
\STATE For each $k\in\{0,\cdots, \text{rank}(\bm L)\}$ evaluate (\ref{eqn:main_bound}) for all pairs $(\tau,\hat{\epsilon}(\tau))$ computed in Step 3 above.
\STATE Choose the minimum value achieved by (\ref{eqn:main_bound}) over all $k$ and $(\tau,\hat{\epsilon}(\tau))$\\ as an estimate for $\|\bm{H}-\hat{\bm{H}}\|$.
\end{algorithmic}
\end{algorithm}

\section{Numerical Examples}
\label{sec:examples}

To investigate various aspects of the proposed bi-fidelity approximation and the associated error estimate, we consider two practically motivated problems.

\subsection{Test Case 1: Heat Driven Cavity Flow}
\label{subsec:cavity}
A practical case for consideration comes from temperature-driven fluid flow in a cavity~\cite{LeMaitre02b, LeMaitre10, LeQuere91, Peng14, Hampton14, Hampton15, Fairbanks2016, Hampton17}, where the QoI is the steady state heat flux along the hot wall; see Figure~\ref{fig:cavity_example}. The left vertical wall has a random temperature $T_h$, referred to as the hot wall, while the right vertical wall, referred to as the cold wall, has a spatially varying stochastic temperature $T_c<T_h$ with constant mean $\bar{T}_c$. Both horizontal walls are adiabatic. The reference temperature and the reference temperature difference are defined as $T_{ref}=\bar{T}_c$ and $\Delta T_{ref}=T_h-\bar{T}_c$, respectively. Under small temperature difference assumption, i.e., the Boussinesq approximation, the normalized governing equations are given by, \cite{LeMaitre02b},
\begin{equation}
\label{eqn:cavity}
\begin{aligned}
&\frac{\partial \bm{u}}{\partial t} + \bm{u}\cdot\nabla\bm{u}=-\nabla p + \frac{\text{Pr}}{\sqrt{\text{Ra}}}\nabla^2\bm{u}+\text{Pr}\Theta\bm{e}_y,\\
& \nabla\cdot\bm{u}=0,\\ 
&\frac{\partial \Theta}{\partial t}+ \nabla\cdot(\bm{u}\Theta)=\frac{1}{\sqrt{\text{Ra}}}\nabla^2\Theta,
\end{aligned}
\end{equation}
where $\bm{e}_y$ is the unit vector $(0,1)$, $\bm{u}=(u,v)$ is velocity vector field, $\Theta=(T- T_{ref})/\Delta T_{ref}$ is normalized temperature, $p$ is pressure, and $t$ is time. Zero velocity boundary conditions on all walls (in both directions) are assumed. For more details on the normalization of the variables in (\ref{eqn:cavity}), we refer the interested reader to \cite{LeQuere91,LeMaitre02b}. Prandtl and Rayleigh numbers are defined, respectively, as $\text{Pr}=\nu/\alpha$ and $\text{Ra}={g}\tau\Delta T_{ref}{L}^3/({\nu}{\alpha})$. Specifically, $L$ is the length of the cavity, ${g}$ is gravitational acceleration, $\nu$ is kinematic viscosity, $\alpha$ is thermal diffusivity, and the coefficient of thermal expansion is given by $\tau$. In this example, we set $g=10$, $L=1$, $\tau = 0.5$, and $\text{Pr}=0.71$. We use a  finite volume method for the discretization of (\ref{eqn:cavity}). 

\subsubsection{Sources of Uncertainty}
\label{subsubsec:Sto_BC}

On the cold wall, a temperature distribution with stochastic fluctuations is applied as
\begin{equation}
  T(x=1,y)=\bar{T}_c+T'(y),
\label{eqn:coldwall}
\end{equation}
where $\bar{T}_c = 100$ is a constant expected temperature. The fluctuation $T'(y)$ is given by the truncated Karhunen-Lo\`eve-like expansion
\begin{equation}
\label{eq:T'}
T'(y) = \sigma_T\sum_{i=1}^{d_T}\sqrt{\lambda_i}\varphi_i(y)\mu_i,
\end{equation}
where $d_T=50$ and $\sigma_T=2$. Here, each $\mu_i$ is assumed to be an independent and identically distributed uniform random variable on $[-1,1]$, and $\{\lambda_i\}_{i=1}^{d_T}$ and $\{\phi_i(y)\}_{i=1}^{d_T}$ are the $d_T$ largest eigenvalues and the corresponding eigenfunctions of the exponential covariance kernel 
\begin{equation}
C_{TT}(y_1,y_2) = \exp{\left(-\frac{\vert y_1-y_2\vert}{l_T}\right)},
\end{equation}
with correlation length $l_T=0.15$. An example of the cold boundary condition is shown in Figure~\ref{fig:cavity_example}. The temperature on the hot wall $T_h$ is also assumed to be random and uniformly distributed over $[105, 109]$. Finally, we consider the viscosity $\nu$ to be uniformly distributed over $[0.004,0.01]$. In total, the dimension of the random input $\bm\mu$ is $52$. The QoI, heat flux along the hot wall, is represented by a vector of heat flux values at the grid points (along the hot wall) corresponding the high- or low-fidelity meshes.

\begin{figure}[H]
\centering
\includegraphics[width = 0.47\textwidth]{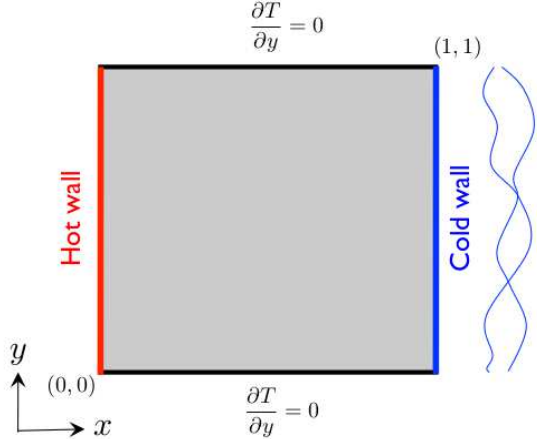}
\caption{Schematic of the heat driven cavity flow problem, reproduced from Figure 5 of~\cite{Fairbanks2016}.}
\label{fig:cavity_example}
\end{figure}
\subsubsection{Bi-Fidelity Approximation and Error Bound Estimate}
\label{subsubsec:sto_baspec}

To compute realizations of the QoI, we use finite volume discretizations of various resolutions to solve (\ref{eqn:cavity}), where the highest resolution, a 256 $\times$ 256 grid point mesh, is used as the high-fidelity model, while relatively coarse meshes of 128 $\times$ 128, 64 $\times$ 64, 32 $\times$ 32, and 16 $\times$ 16 grid points are used for the corresponding low-fidelity model. All meshes are spatially uniform. We consider bi-fidelity approximations of various ranks $r$, corresponding to the number of basis vectors in the column skeleton matrices $\bm{H}(r)$ and $\bm{L}(r)$ introduced in Section~\ref{sec:theorem}. We also consider bounds derived from computations of optimal pair $(\tau, \hat{\epsilon}(\tau))$ as explained in Section~\ref{sec:error_estimate}. These estimates are drawn from $n$ randomly generated high-fidelity samples, i.e., $n$ randomly selected columns of $\bm H$, and we consider estimates computed from various sample sizes $n$. 

Figure~\ref{fig:bi_fi_realization_heatflux} displays realizations of heat flux along the hot wall for a random sample of $\bm\mu$, obtained by the high-fidelity, low-fidelity, and rank $r=10$ bi-fidelity models. We observe the close agreement between the bi- and high-fidelity solutions even when the 16 $\times$ 16 mesh is used as the low-fidelity model. Figure~\ref{fig:cavity_hists} provides four histograms of the low- and bi-fidelity error to gauge the performance of all 100 QoI realizations. For (a)-(d), which use increasingly refined meshes for the low-fidelity model, we see an improvement of the bi-fidelity performance over that of the associated low-fidelity model. 

We next consider the error bound efficacy of using  $\hat{\epsilon}(\tau)$ from (\ref{eq:eps_tau_hat_def}), in place of $\epsilon(\tau)$ from (\ref{eqn:eps_tau_def}), in Theorem~\ref{thm:algorithm}. This efficacy is defined as the ratio of the error estimate over the actual error in the approximation of $\bm H$ so that an efficacy of $1$ implies the error estimate is exact, and an efficacy greater than $1$ implies that the computed error estimate is in fact a bound. Figure~\ref{fig:cavity_err_eff} shows this error efficacy for various low-fidelity models, approximation ranks $r$, and sample sizes $n$. For each pair $(n,r)$ the reported error ratio is the average of $30$ error ratios, each computed from independent sets of $n$ high-fidelity samples. We note that the error efficacy is greater than 1 whenever the sample size $n$ is greater than the approximation rank $r$, and that, as long as this condition holds, the error efficacy is less than 10, implying that the error bound is off by less than one order of magnitude. This is of great utility in determining either a sample size $n$ for a given approximation rank $r$ or, conversely, determining an admissible rank of approximation for a given sample size for which the computed error estimate is still valid.

\begin{figure}[H]
\centering
\subfloat[16 $\times$ 16 Low-Fidelity Mesh]{\includegraphics[width = 0.49\textwidth]{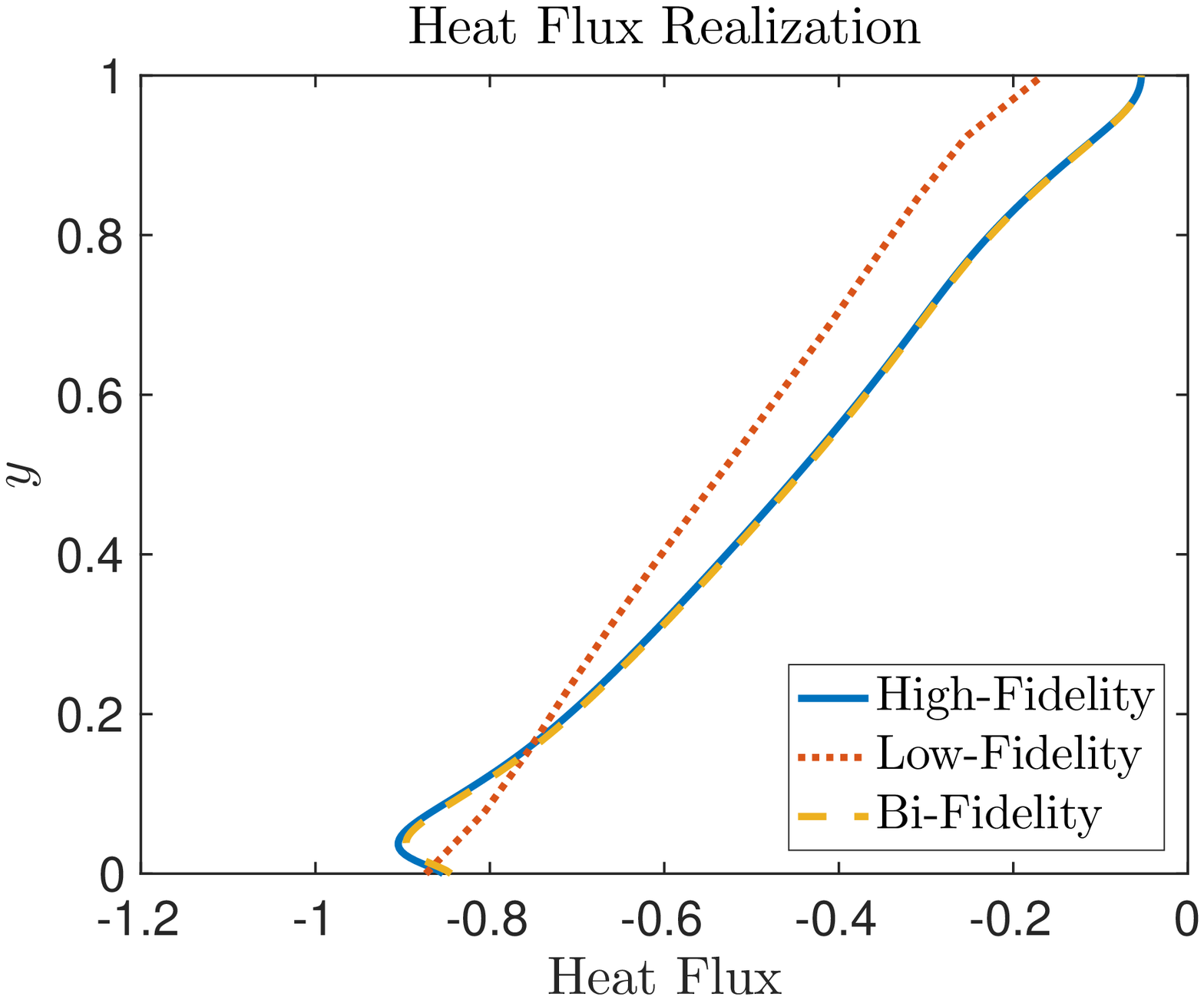}}\hspace{1mm}
\subfloat[32 $\times$ 32 Low-Fidelity Mesh]{\includegraphics[width = 0.49\textwidth]{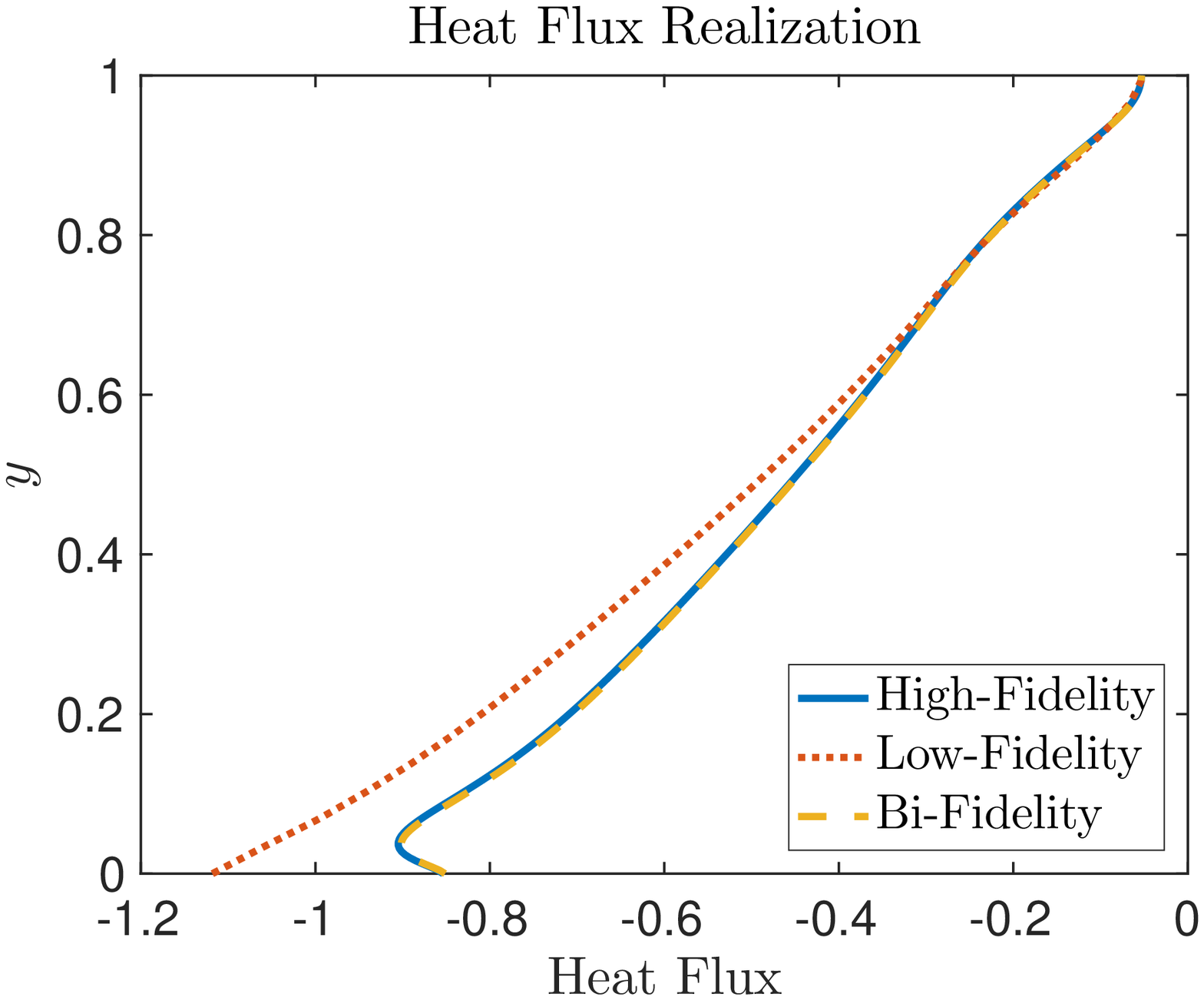}}\\
\subfloat[64 $\times$ 64 Low-Fidelity Mesh]{\includegraphics[width = 0.49\textwidth]{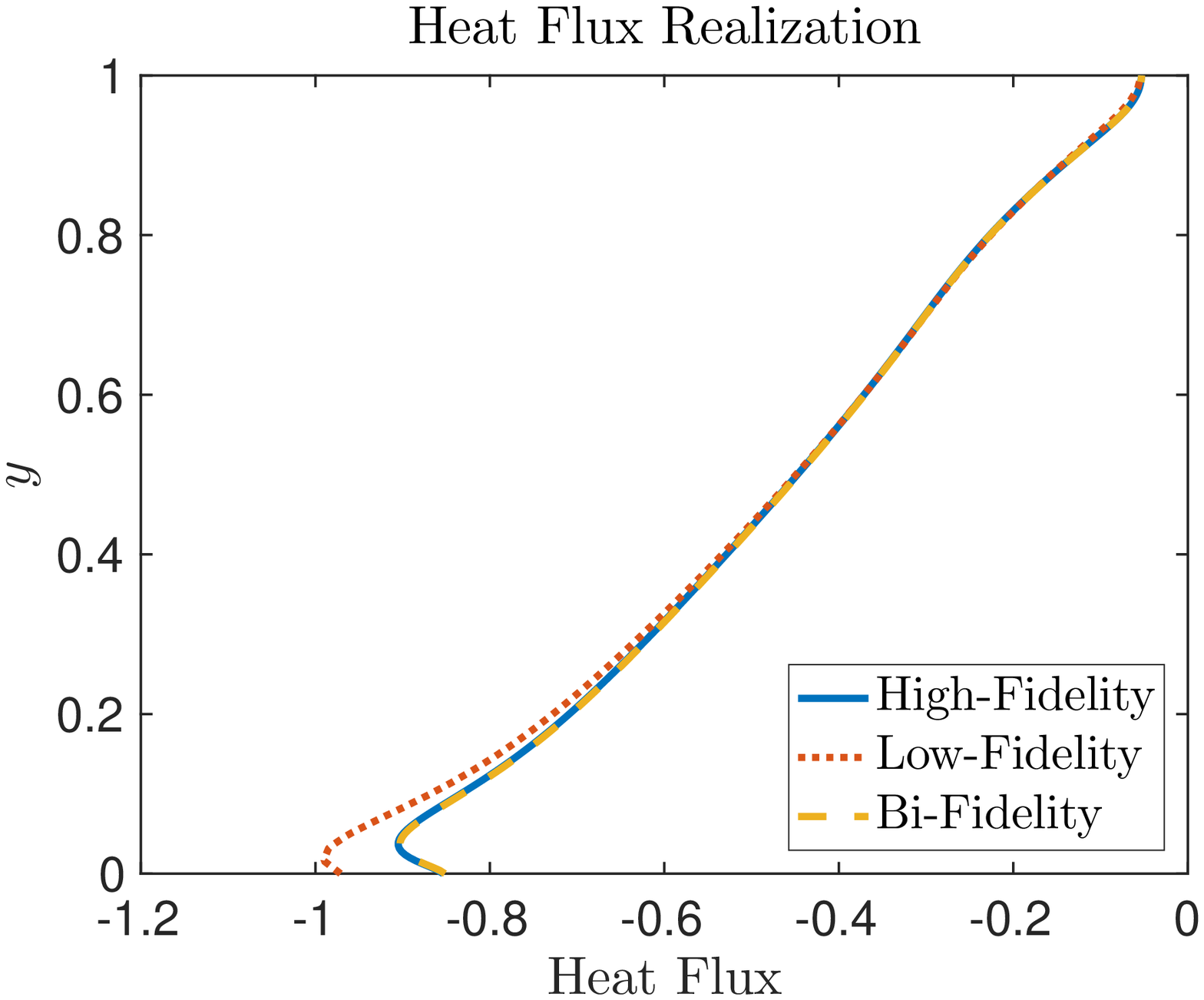}}\hspace{1mm}
\subfloat[128 $\times$ 128 Low-Fidelity Mesh]{\includegraphics[width = 0.49\textwidth]{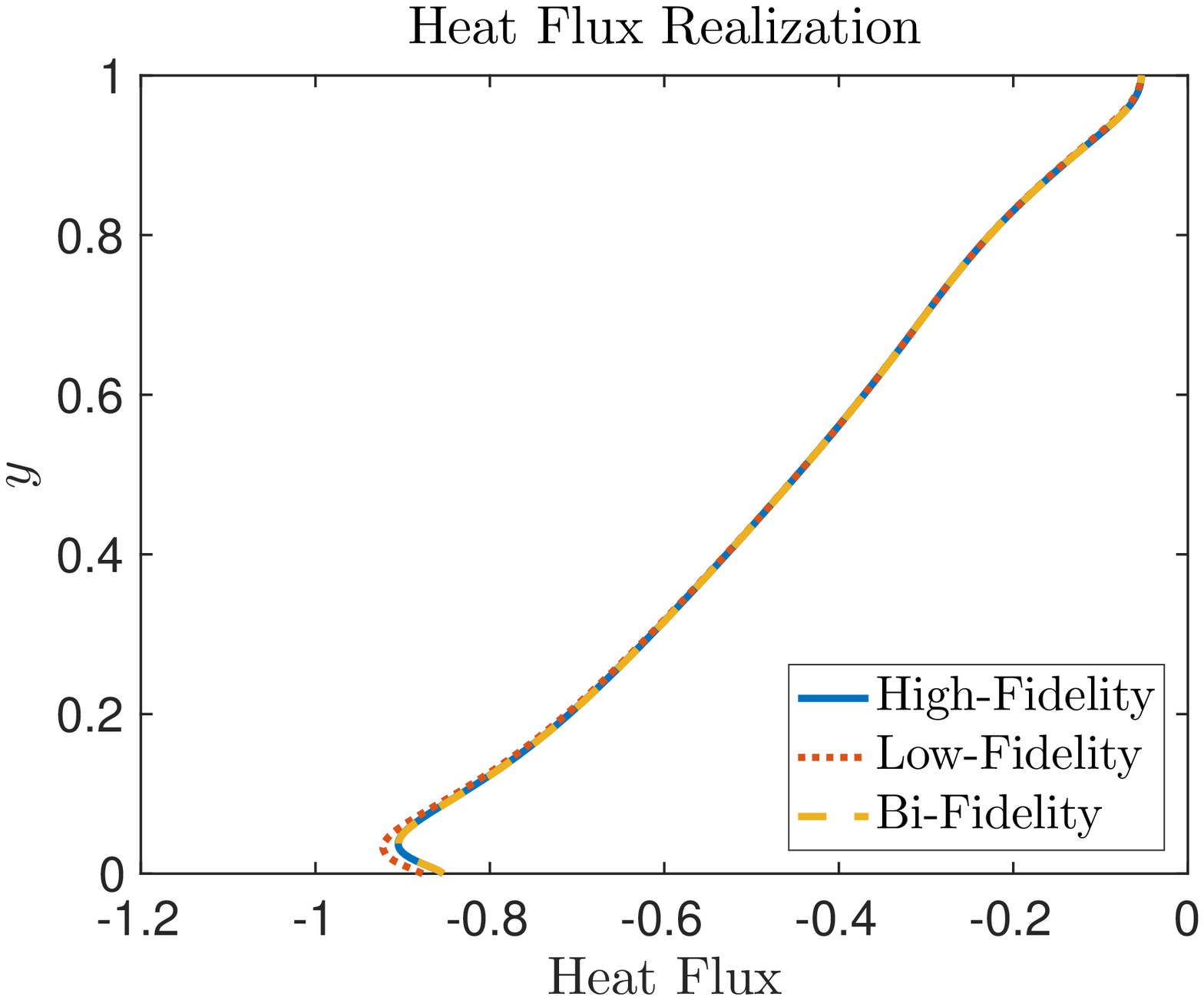}}
\caption{Realizations of heat flux along the hot wall for a randomly selected input $\bm\mu$. Shown are the low-fidelity, high-fidelity, and rank $r=10$ bi-fidelity estimates for the various low-fidelity models.}
\label{fig:bi_fi_realization_heatflux}
\end{figure}
\begin{figure}[H]
\centering
\subfloat[16 $\times$ 16 Low-Fidelity Mesh]{\includegraphics[width = 0.49\textwidth]{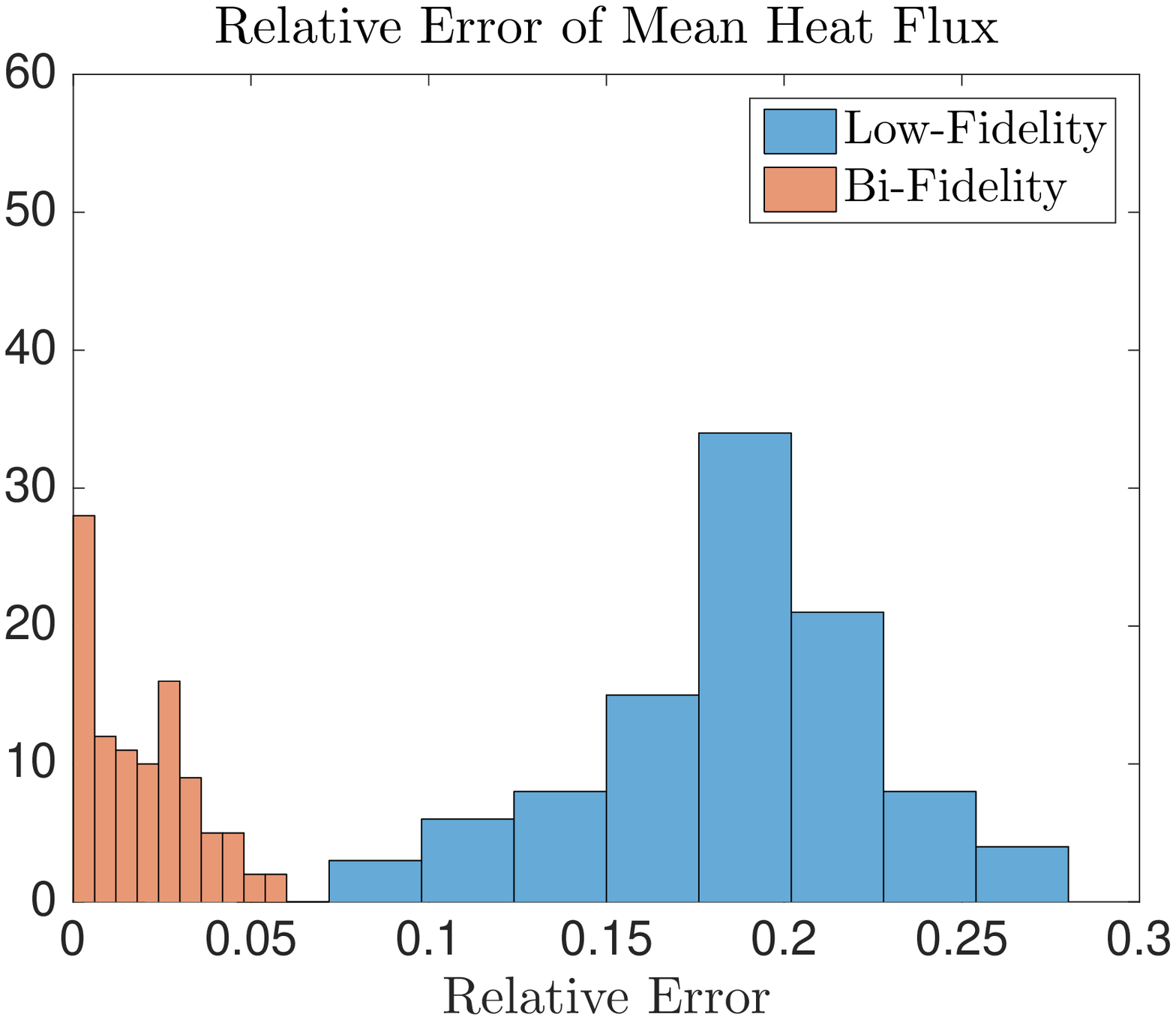}}\hspace{1mm}
\subfloat[32 $\times$ 32 Low-Fidelity Mesh]{\includegraphics[width = 0.49\textwidth]{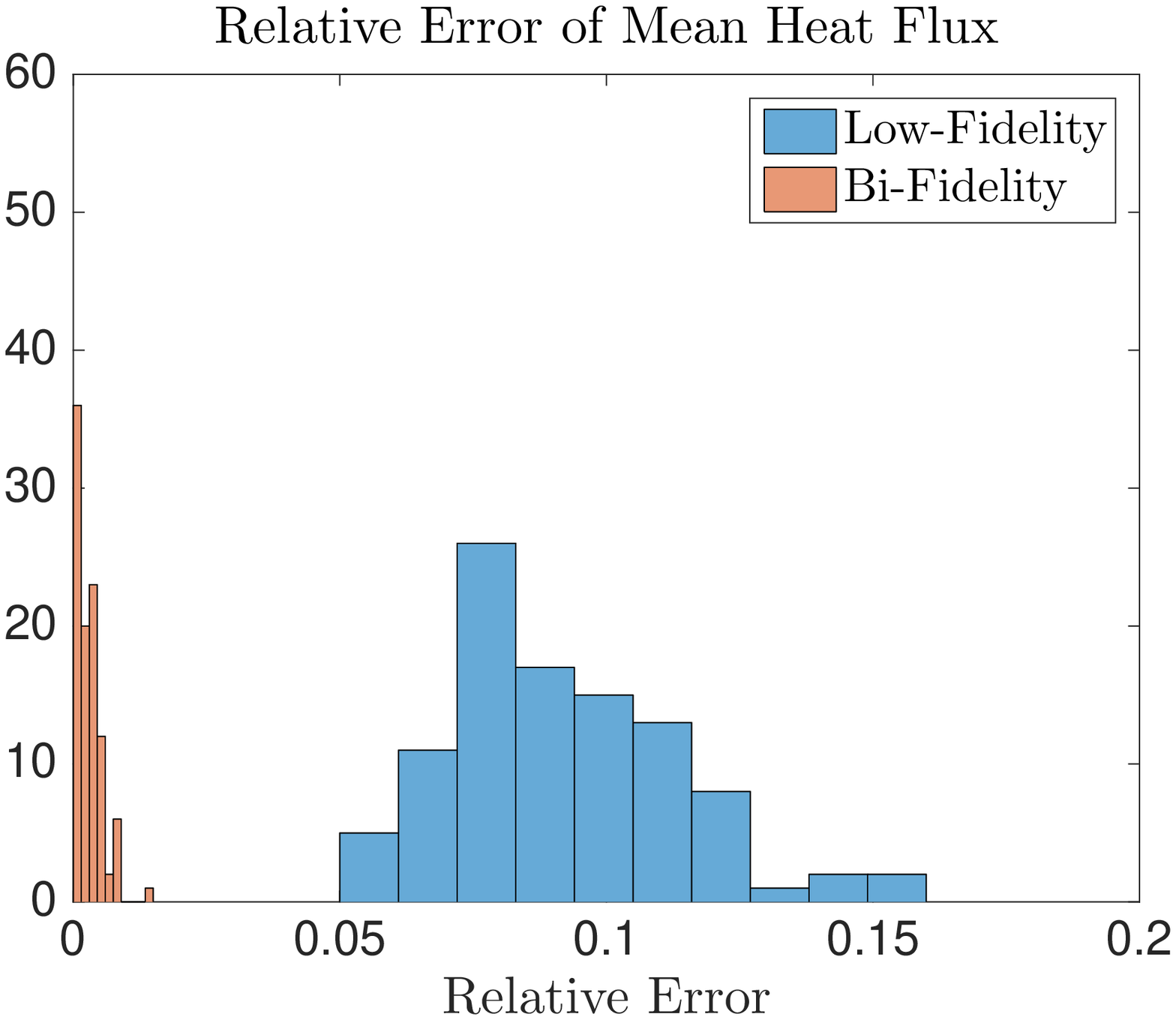}}\\
\subfloat[64 $\times$ 64 Low-Fidelity Mesh]{\includegraphics[width = 0.49\textwidth]{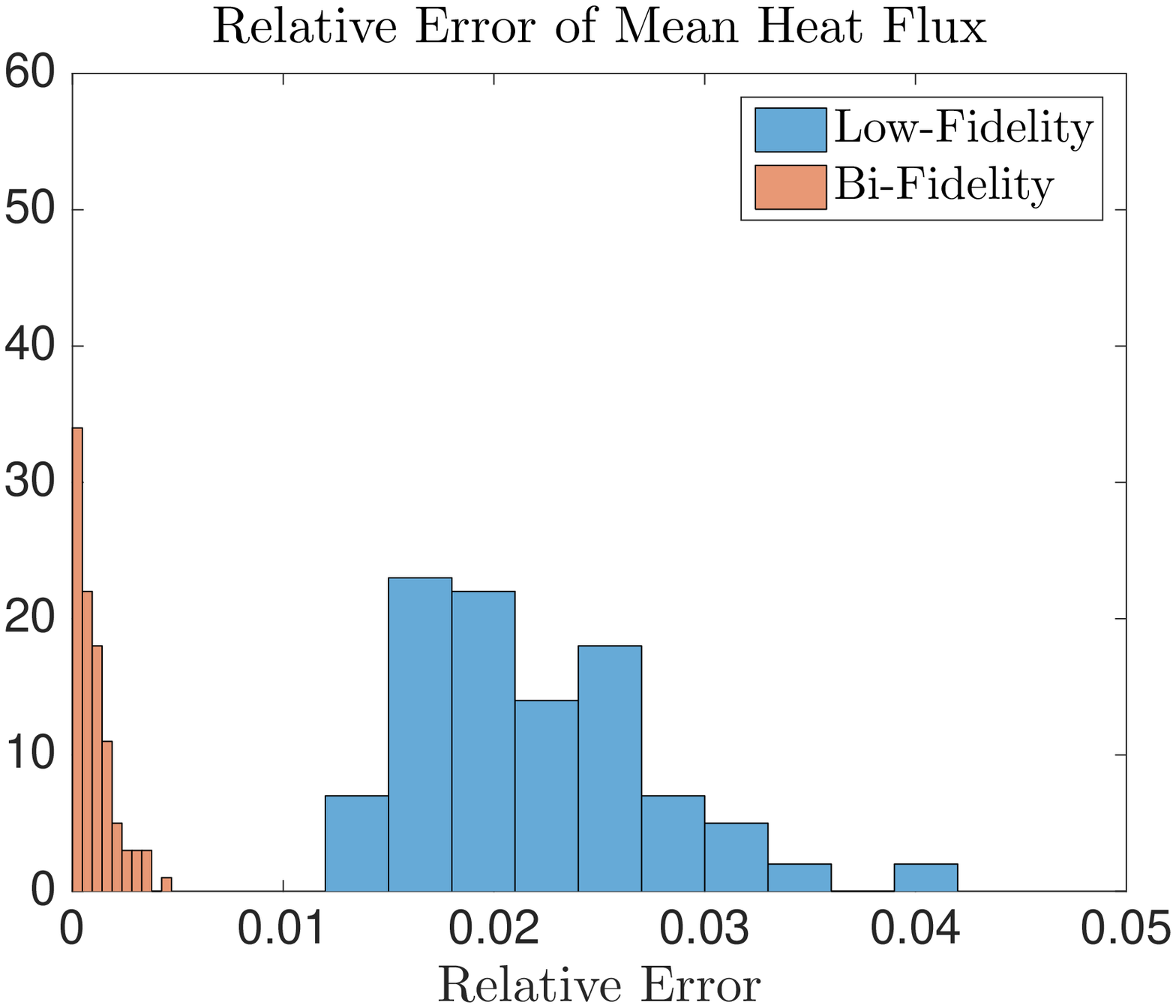}}\hspace{1mm}
\subfloat[128 $\times$ 128 Low-Fidelity Mesh]{\includegraphics[width = 0.49\textwidth]{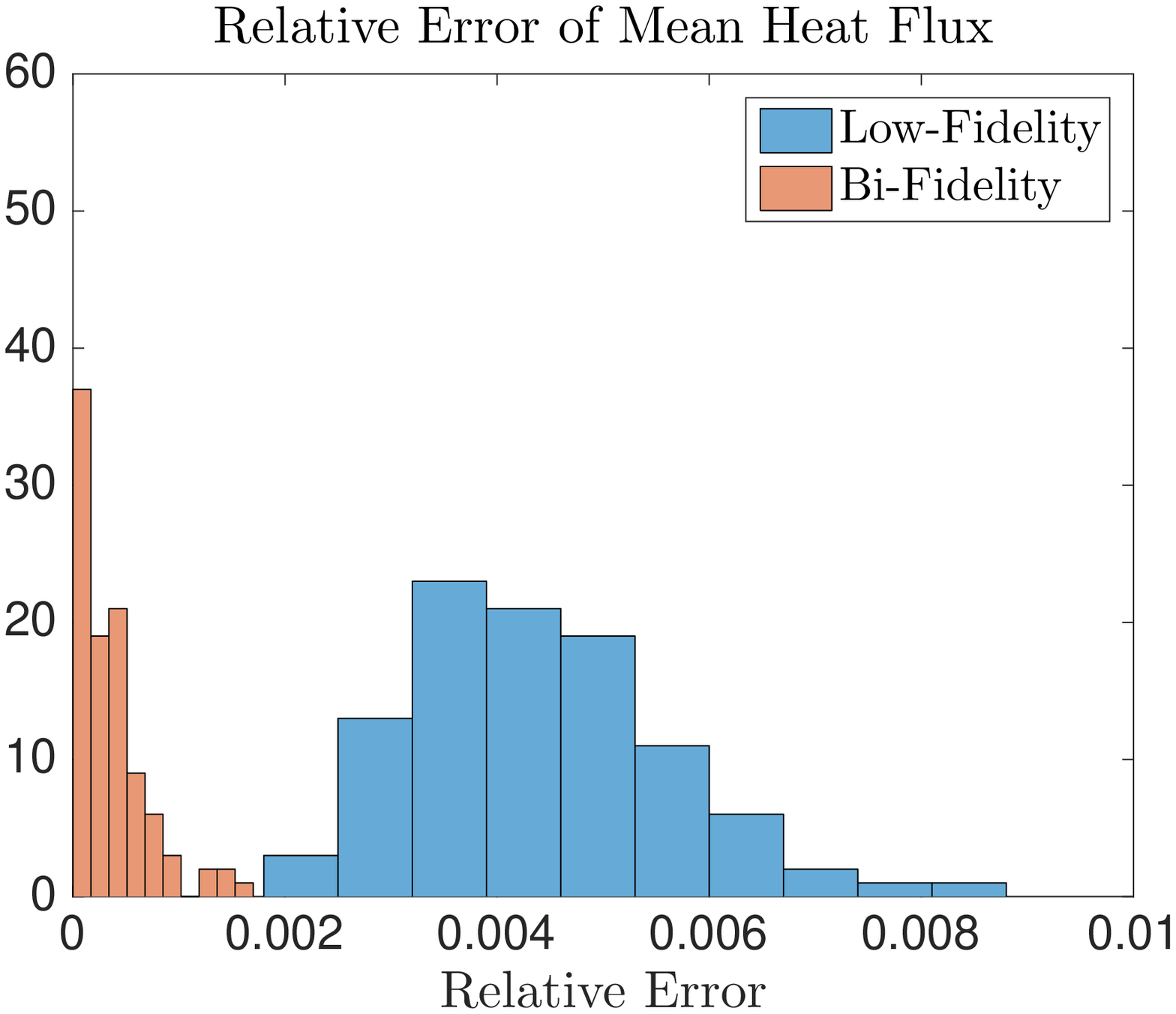}}
\caption{Histograms of low-fidelity and rank $r=10$ bi-fidelity errors normalized by $\Vert \bm{H}\Vert$ for various low-fidelity models.}
\label{fig:cavity_hists}
\end{figure}
\begin{figure}[H]
\centering
\subfloat[16 $\times$ 16 Low-Fidelity Mesh]{\includegraphics[width = 0.47\textwidth]{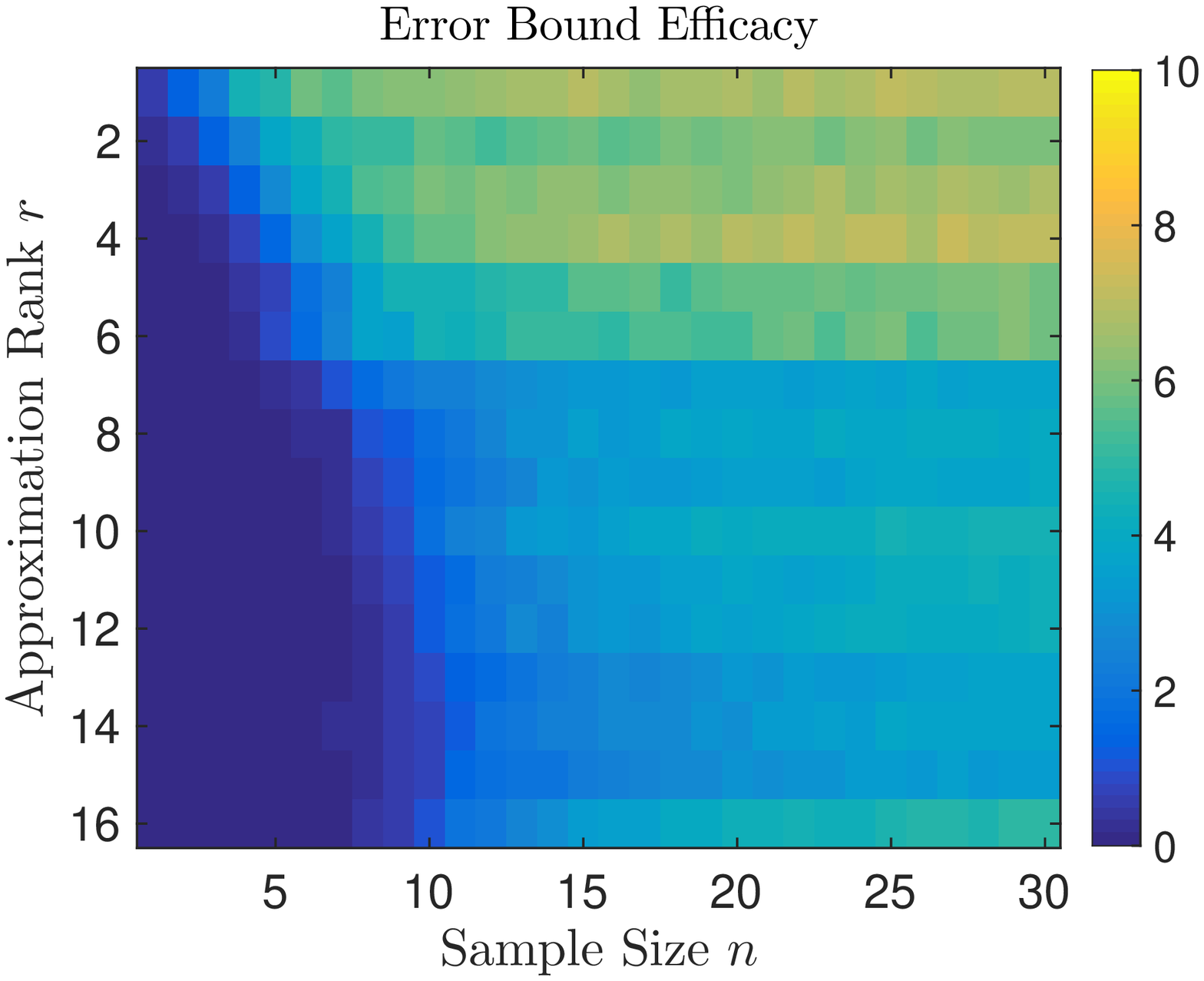}}\hspace{1mm}
\subfloat[32 $\times$ 32 Low-Fidelity Mesh]{\includegraphics[width = 0.47\textwidth]{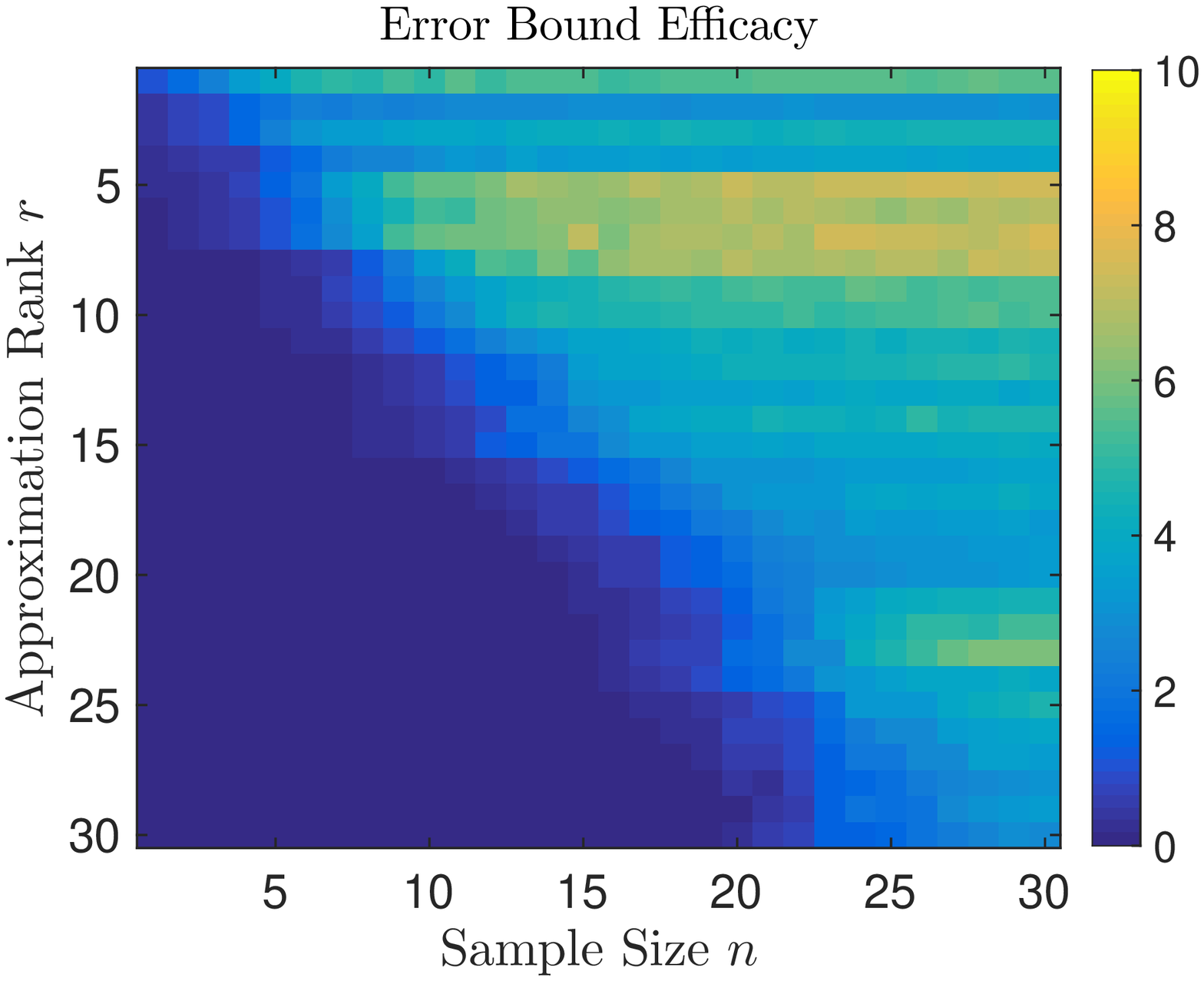}}\\
\subfloat[64 $\times$ 64 Low-Fidelity Mesh]{\includegraphics[width = 0.47\textwidth]{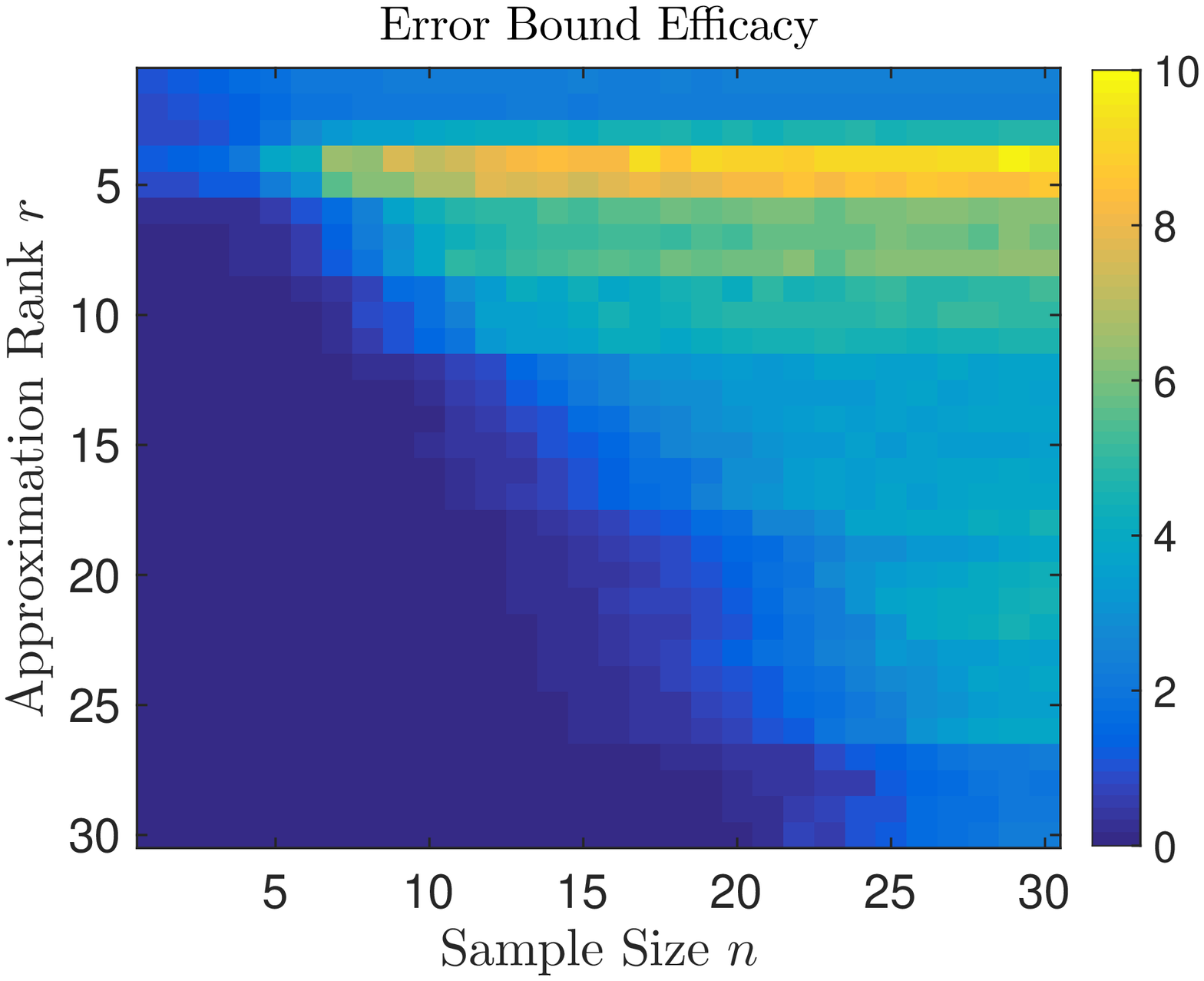}}\hspace{1mm}
\subfloat[128 $\times$ 128 Low-Fidelity Mesh]{\includegraphics[width = 0.47\textwidth]{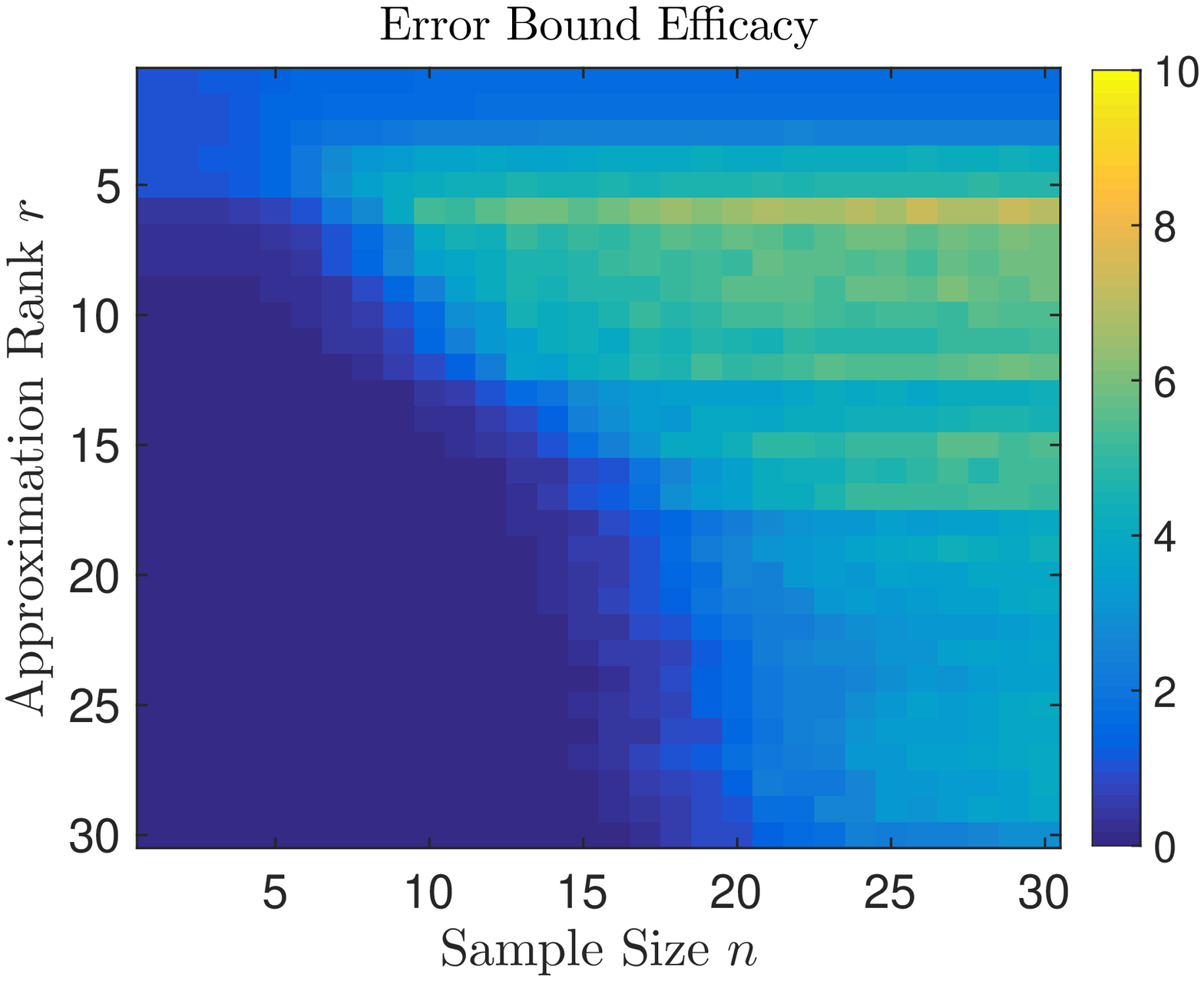}}
\caption{Identification of error bound efficacy, i.e., the (average) ratio between the error estimated from (\ref{eqn:main_bound}) and the true error, for various low-fidelity meshes, approximation ranks $r$, and sample sizes $n$. For each pair $(n,r)$ the reported error ratio is the average of $30$ ratios, each computed from an independent set of $n$ high- and low-fidelity samples via Algorithm \ref{alg:bound}.}
\label{fig:cavity_err_eff}
\end{figure}

To address the estimation of (\ref{eqn:main_bound}) numerically, we consider in Figure~\ref{fig:eps_tau_funcs} how $\hat{\epsilon}(\tau)$ in (\ref{eq:eps_tau_hat_def}) behaves as a function of $\tau$ for the case of $r=10$ and $n=20$. We further mark points that optimize (\ref{eqn:main_bound}). We note the reduction in the optimal values of $\tau$ and $\hat{\epsilon}(\tau)$ as the mesh of the low-fidelity model is refined.
\begin{figure}[H]
\centering
\includegraphics[width = 0.47\textwidth]{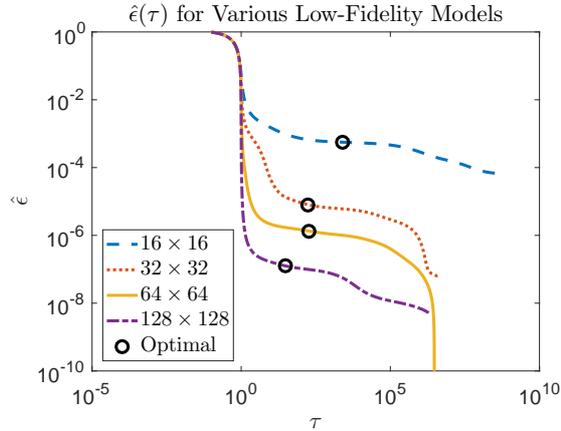}
\caption{$\hat{\epsilon}(\tau)$ from (\ref{eq:eps_tau_hat_def}) as a function of $\tau$ for different low-fidelity model meshes. In all cases the approximation rank $r=10$ and sample size $n=20$. Values which optimize (\ref{eqn:main_bound}) are marked by black circles.}
\label{fig:eps_tau_funcs}
\end{figure}

We can also explore the behavior of $\rho_k(\tau)$ as a function of $\tau$ in \eqref{eq:error_func}, noting that minimizing over $\tau$ and $k$ yields our main estimate \eqref{eqn:main_bound_rho}.
%$\tau$ defined by
%
%\begin{align}
%\label{eq:error_func}
%\rho_k(\tau) &\coloneqq \left[\left(1 + \|\bm{C}_L\|\right)\sqrt{\tau\sigma^2_{k+1} + \hat{\epsilon}(\tau)} + \|\bm{L}-\hat{\bm{L}}\|\sqrt{\tau + \hat{\epsilon}(\tau)\sigma_k^{-2}}\right],
%\end{align}
%?????????????????????????????????????????????????????
%
%and note that minimizing $\rho_k(\tau)$ over $\tau$ and $k$ would recover the estimate (\ref{eqn:main_bound}). 
This function is seen for different low-fidelity model meshes in Figure~\ref{fig:error_funcs}. We observe that as the low-fidelity meshes are refined, $\tau$ converges to $1$, corresponding to a broadly defined convergence of low-fidelity models to the high-fidelity model. We also note, that there is not a high level of sensitivity with regards to $\tau$, and that a large range of $\tau$ produces comparable bounds. Figure~\ref{fig:error_funcs} also illustrates the benefit that the optimization of $\tau$ need not be particularly accurate, and that the general shape of these curves is smooth enough that derivative-based optimization can be employed, limiting the number of $\tau$ for which $\hat{\epsilon}(\tau)$ needs to be evaluated. 

It is worth highlighting that the heuristic employed here is that the size of the matrices used in $\hat{\epsilon}(\tau)$ need only depend on the approximation rank $r$ and scale nearly as such. As a result, the cost of evaluating $\hat{\epsilon}(\tau)$ will typically not be prohibitively large, as illustrated in Figure~\ref{fig:error_sample_size} by the rapid convergence of the optimal $\hat{\epsilon}(\tau)$ with respect to $n$. Finally, Figure~\ref{fig:error_ests} shows the convergence of the $r = 10$ bi-fidelity solution as a function of the low-fidelity mesh size. The error bound estimate -- relative to $\bm H$ -- is displayed for two sample sizes of $n$, and the low-fidelity error -- also relative to $\bm H$ -- is provided for comparison. Notice that, aside for a single point from the coarsest low-fidelity model, all the estimated errors are below the corresponding low-fidelity errors. In addition, the error bound estimate decays with the bi-fidelity error, as the mesh is refined.

\begin{figure}[H]
\centering
\subfloat[Sample Size $n=12$]{\includegraphics[width = 0.49\textwidth]{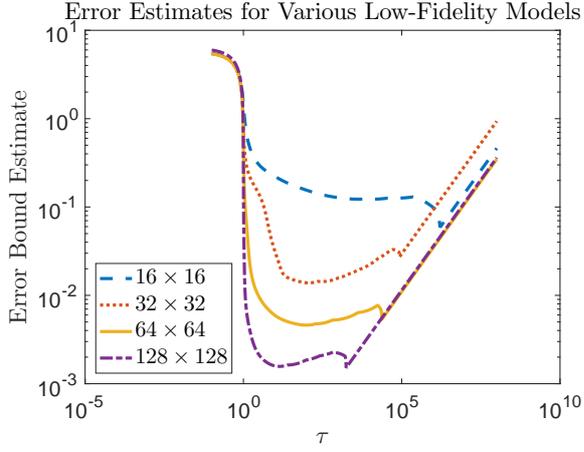}}\hspace{1mm}
\subfloat[Sample Size $n=20$]{\includegraphics[width = 0.49\textwidth]{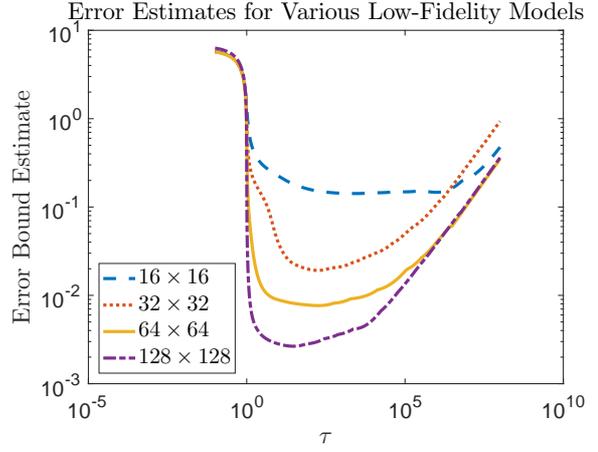}}
\caption{Minimum (over $k$) of error function $\rho_k(\tau)$ in (\ref{eq:error_func}) for various low-fidelity models using an approximation rank $r=10$. Error is normalized by $\Vert\bm H\Vert$.}
\label{fig:error_funcs}
\end{figure}
\begin{figure}[H]
\centering
\subfloat[16 $\times$ 16 Low-Fidelity Mesh]{\includegraphics[width = 0.48\textwidth]{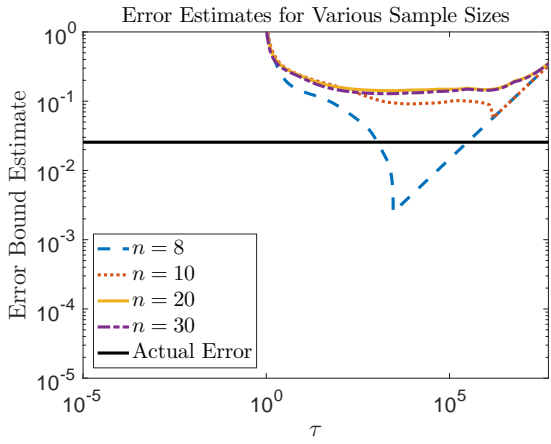}}\hspace{1mm}
\subfloat[128 $\times$ 128 Low-Fidelity Mesh]{\includegraphics[width = 0.48\textwidth]{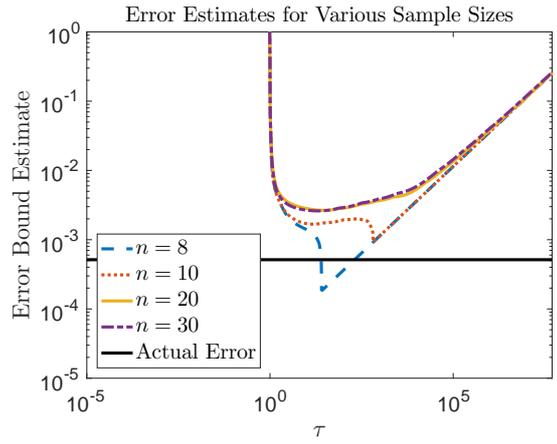}}
\caption{Minimum (over $k$) of error function $\rho_k(\tau)$ in (\ref{eq:error_func}) for for different low-fidelity models using different sample sizes $n$. All bi-fidelity solutions are of rank $r=10$, and error is normalized by $\Vert\bm H\Vert$.}
\label{fig:error_sample_size}
\end{figure}
\begin{figure}[H]
\centering
\includegraphics[width = 0.47\textwidth]{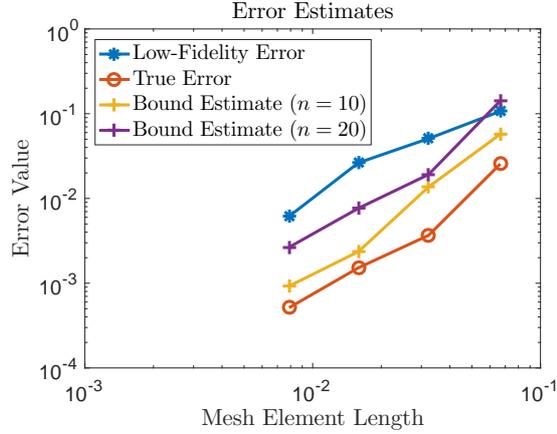}
\caption{Estimate of (\ref{eqn:main_bound}) normalized by $\Vert \bm H\Vert$ for various low-fidelity models and using different sample sizes $n$.}
\label{fig:error_ests}
\end{figure}
\subsection{Test Case 2: Composite Beam}
\label{subsec:beam}

For the second test case, we consider the deformation of a plane-stress, cantilever beam with composite cross section and hollow web, as shown in Figure~\ref{fig:beam_schematics}. The Young's moduli of the three components of the cross section as well as the intensity of the applied distributed force on the beam are assumed to be uncertain, and are modeled by independent uniform random variables. Table~\ref{table:parameters} provides a summary of the parameters of this model along with the description of the uncertain inputs. Here, the QoI is the vertical displacement of the top cord and, in particular, its maximum occurring at the free end. To construct the bi-fidelity approximation, we use realizations of the vertical displacement over the entire top cord given by the low- and high-fidelity models.

\begin{figure}[H]
\centering
\includegraphics[trim = 17mm 70mm 20mm 70mm, clip, width = 0.8\textwidth]{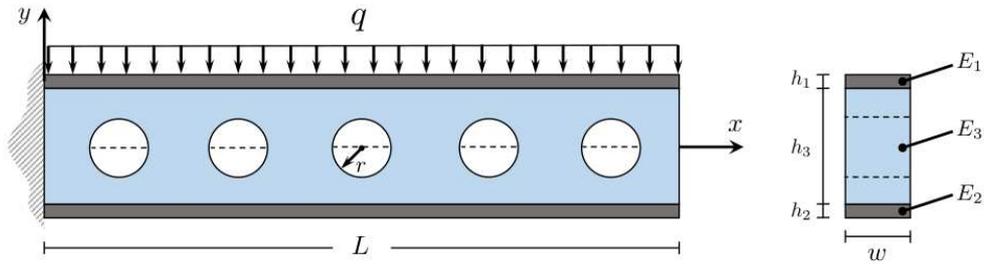}
\caption{Schematic of the cantilever beam (left) and its composite cross section (right).}
\label{fig:beam_schematics}
\end{figure}

\begin{figure}[H]
\centering
\includegraphics[trim = 0mm 80mm 0mm 70mm, clip, width = 0.8\textwidth]{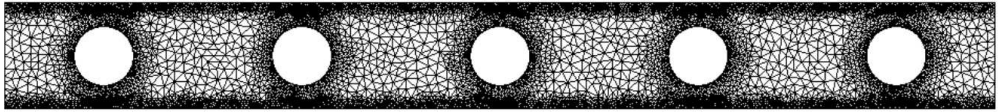}
\caption{Finite element mesh used for high-fidelity simulation of the vertical displacement.}
\label{fig:beam_mesh}
\end{figure}

\begin{table}[htb]
\caption{Description of the parameters of the composite cantilever beam model. The center of the holes are at $x=\{5,15,25,35,45\}$.} 

\vspace{.3cm}
\resizebox{\textwidth}{!}{
\centering
\begin{tabular}{ c c c c c c c c c c}   
\hline
$L$ & $h_1$ & $h_2$ & $h_3$ & $w$ & $r$ & $q$ & $E_1$ & $E_2$ & $E_3$ \\
\hline 
50  &  0.1 & 0.1 & 5 &  1 &  1.5 & $U[9,11]$ & $U[0.9e6,1.1e6]$  &  $U[0.9e6,1.1e6]$ & $U[0.9e4,1.1e4]$  \\
\hline 
\end{tabular} 
}
\label{table:parameters} 
\end{table}

Unlike in the previous test case where the low-fidelity model simulated the same physical problem but on a coarser mesh, here the low-fidelity model solves a simpler physical model. Specifically, the high-fidelity model is based on a finite element (FE) discretization of the beam using a triangular mesh with linear elements; see Figure~\ref{fig:beam_mesh}. The mesh is fine enough to describe the web's hollow geometry and achieve mesh independent predictions of the vertical displacement at the free end. The FE simulations are performed using FEniCS \cite{LoggMardalEtAl2012a}. The low-fidelity model is based on Euler-Bernoulli beam theory, where vertical cross sections are assumed to remain planes throughout the deformation and, thus, the shear deformation of the web is ignored. Additionally, we simplify the geometry of the beam by ignoring the circular holes. These simplifications, while making the calculations straightforward, lead to {\it inaccurate} low-fidelity predictions of the displacement (Figure~\ref{fig:bi_fi_realization_bar}), as the web experiences considerable shear deformation due to the presence of the circular holes.

Following the Euler-Bernoulli theorem, the vertical displacement of the beam $u(x)$ is given by
\begin{equation}
\label{eqn:EB_disp}
EI\frac{\text{d}^4 u(x)}{\text{d}x^4}=-q,
\end{equation}
where $E$ and $I$ are, respectively, the Young's modulus and the moment of inertia of an {\it equivalent} cross section (about its centroid axis) consisting of a single material. Specifically, we let $E=E_3$ and consider an equivalent cross section in which the width of the top and bottom sections are set to $w_1=(E_1/E_3)w$ and $w_2=(E_2/E_3)w$, respectively, while all other dimensions remain as before; see Figure~\ref{fig:beam_schematics} and Table~\ref{table:parameters}. The solution to (\ref{eqn:EB_disp}) for the considered beam is then given by 
\begin{equation}
\label{eqn:EB_disp_sol}
u(x)=-\frac{qL^4}{24EI}\left(\left(\frac{x}{L}\right)^4-4\left(\frac{x}{L}\right)^3+6\left(\frac{x}{L}\right)^2\right).
\end{equation}

Notice that all realizations of $u(x)$ depend on $x$ identically, albeit with different coefficient $qL^4/24EI$. This implies that the low-fidelity solution is theoretically of rank $r=1$ and that only one high-fidelity realization of $u(x)$ is needed to generate the bi-fidelity approximation. Stated differently, the bi-fidelity approximation cannot be improved by increasing the rank beyond $r=1$. This has significant implications for $\epsilon(\tau)$ in (\ref{eqn:eps_tau_def}) and the corresponding interpretation of Theorem~\ref{thm:algorithm}. Specifically, with regards to (\ref{eqn:main_bound}) we note that the only valid choice of $r$ is $r=1$; therefore, both $\sigma_{r+1}$ and $\|\bm{L}-\hat{\bm{L}}\|$ are precisely zero so that
\begin{align*}
\|\bm{H}-\hat{\bm{H}}\| &\le \mathop{\min}\limits_{\tau}\left(1 + \|\bm{C}_L\|\right)\sqrt{\epsilon(\tau)}.
\end{align*}
As $\epsilon$ is a decreasing function of $\tau$, we note that the minimizing value of $\tau$ occurs in the limit as $\tau\rightarrow\infty$. This can be numerically observed from Figure~\ref{fig:eps_tau_funcs_bar}. Even though the low-fidelity model is of a minimal rank, the high-fidelity model is itself reasonably well approximated by a representation of rank $r=1$. This is seen in Figure~\ref{fig:bi_fi_realization_bar} for two randomly selected choices of input parameters, illustrating a considerably closer agreement between the bi- and high-fidelity samples relative to the low- and high-fidelity samples. Figure~\ref{fig:beam_hist} shows histograms of the low- and bi-fidelity errors of the maximum displacement, where the average error achieved by the bi-fidelity model is roughly an order of magnitude smaller. In Figure~\ref{fig:beam_eps_tau_funcs_bar} we see that the error bound stabilizes by around $7$ samples and overestimates the true error by a factor $9$. With regards to the efficacy of the bound (\ref{eqn:main_bound}) and its dependency on $\tau$ as a function of the sample size $n$ used to determine an optimal $\hat{\epsilon}(\tau)$, we see in Figure~\ref{fig:error_ests_bar_2} that the bound estimate, and specifically its dependence on $\tau$, is robust to $n$, converging quickly with respect to $n$.

\begin{figure}[H]
\centering
\includegraphics[width = 0.47\textwidth]{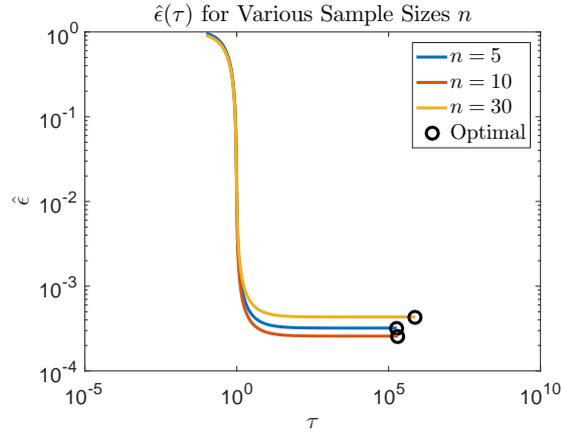}
\caption{$\hat{\epsilon}(\tau)$ from (\ref{eq:eps_tau_hat_def}) as a function of $\tau$ for different sample sizes $n$. Values which optimize (\ref{eqn:main_bound}) for a rank $r=1$ approximation are marked.}
\label{fig:eps_tau_funcs_bar}
\end{figure}
\begin{figure}[H]
\centering
\subfloat[Realization 1]{\includegraphics[width = 0.47\textwidth]{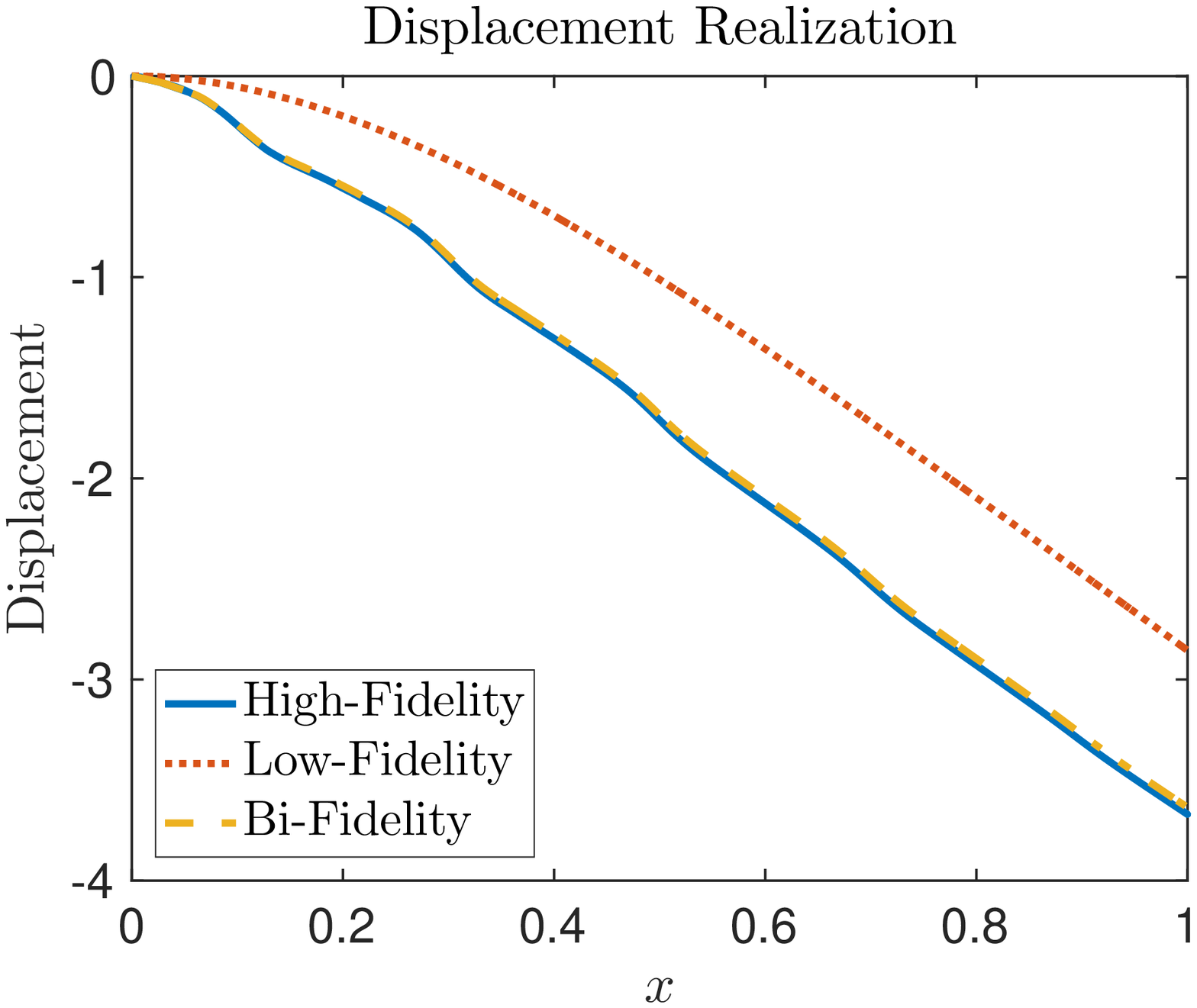}}
\subfloat[Realization 2]{\includegraphics[width = 0.47\textwidth]{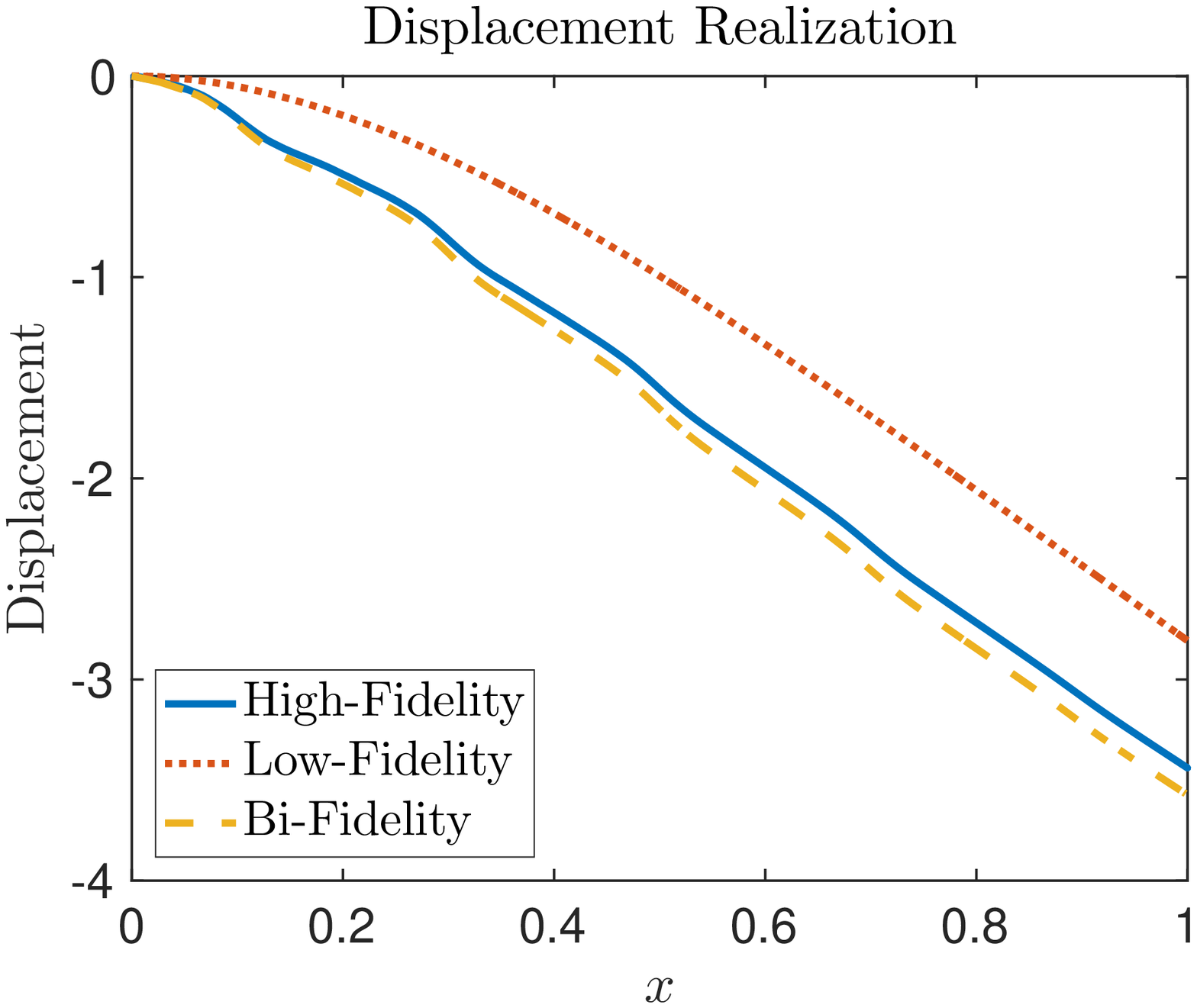}}
\caption{Realizations of vertical displacement for two randomly selected input parameters $\bm \mu$. Shown are the low-fidelity, high-fidelity, and rank $r=1$ bi-fidelity estimates.}
\label{fig:bi_fi_realization_bar}
\end{figure}
\begin{figure}[H]
\centering
\includegraphics[width = 0.47\textwidth]{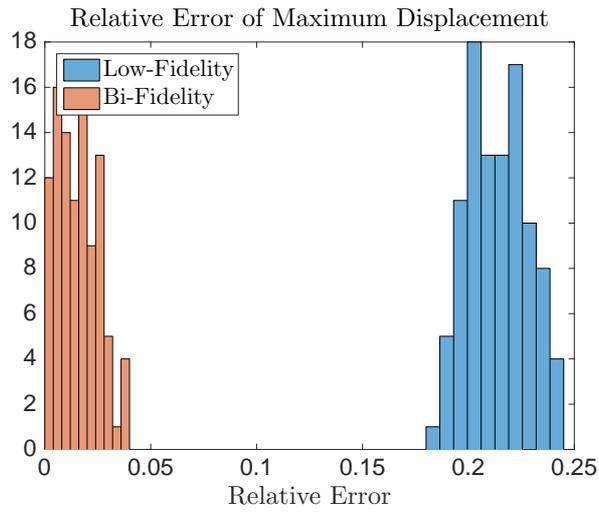}
\caption{Histogram for error in the maximum displacement for low- and bi-fidelity models.}
\label{fig:beam_hist}
\end{figure}

\begin{figure}[H]
\centering
\includegraphics[width = 0.47\textwidth]{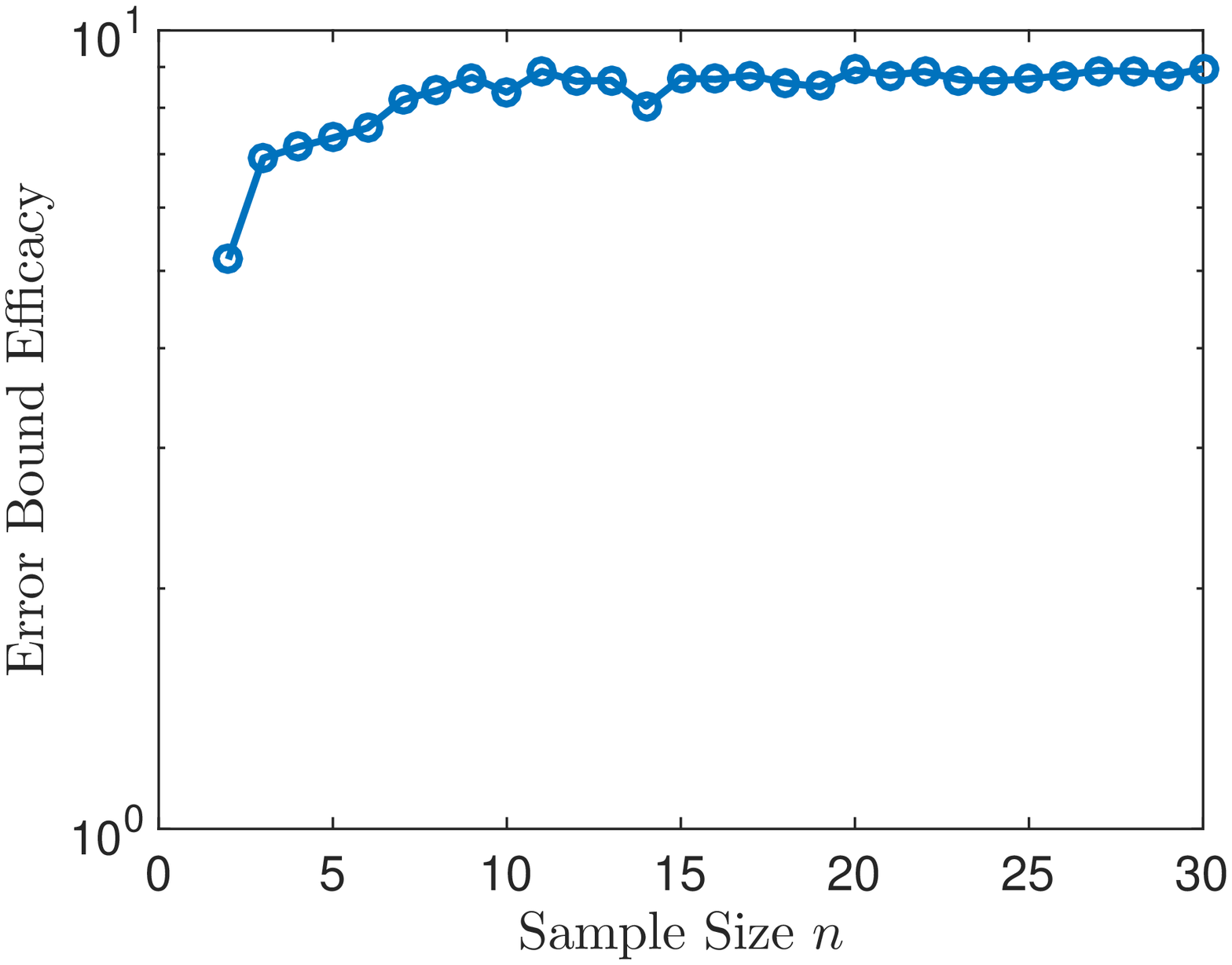}
\caption{Identification of error bound efficacy, i.e., the (average) ratio between the error estimated from (\ref{eqn:main_bound}) and the true error, for various sample sizes $n$. For each $n$ the reported error ratio is the average of $30$ ratios, each computed from an independent set of $n$ high- and low-fidelity samples via Algorithm \ref{alg:bound}.}
\label{fig:beam_eps_tau_funcs_bar}
\end{figure}
\begin{figure}[H]
\centering
\includegraphics[width = 0.5\textwidth]{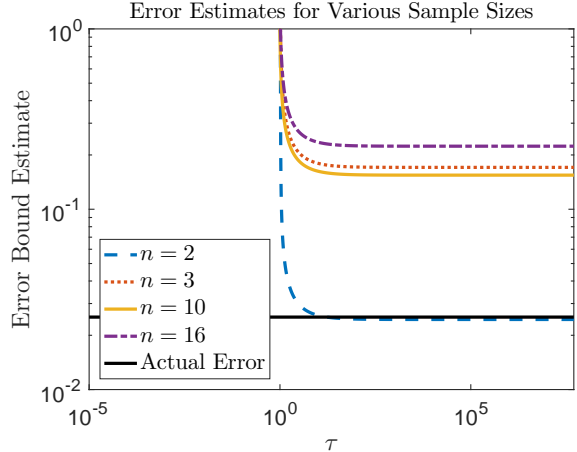}
\caption{Minimum (over $k$) of error function $\rho_k(\tau)$ in (\ref{eq:error_func})  using different sample sizes $n$. Error is normalized by $\Vert\bm H\Vert$.}
\label{fig:error_ests_bar_2}
\end{figure}
\section{Conclusions}
\label{sec:conc}

This work is concerned with the error analysis of a bi-fidelity, low-rank approximation technique for a reduced order solution of problems with stochastic or parametric inputs. The bi-fidelity method relies on the solution to a cheaper, lower-fidelity model of the intended expensive, high-fidelity model of the problem in order to identify a reduced basis and an interpolation rule for approximating the high-fidelity solution in its inputs. A novel and practical error bound is presented for this bi-fidelity approximation that is shown to not require many high-fidelity samples beyond the rank of the approximation, thus making the bound an effective tool for practical application of this approximation. The bi-fidelity estimate and the associated error bound are by construction non-intrusive, i.e., sample-based, and thus can use legacy solvers in a black box fashion. Some aspects of the bi-fidelity approximation and specifically the efficacy of the derived error bound are demonstrated on two numerical examples from fluid (non-linear) and solid (linear) mechanics. 

\section*{Acknowledgements}

The work of JH and AN was supported by the DARPA EQuiPS project, award number N660011524053. The work of HF was supported by the United States Department of Energy under the Predictive Science Academic Alliance Program 2 (PSAAP2) at Stanford University. This material is based upon work of AD supported by the U.S. Department of Energy Office of Science, Office of Advances Scientific Computing Research, under Award Number DE-SC0006402, and NSF grant CMMI-145460. AN was partially supported by AFOSR FA9550-15-1-0467.
\label{sec:acknow}

\appendix
\section{Theoretical Exposition}
\label{sec:theory}

This section is devoted to proving our main theoretical result, Theorem \ref{thm:algorithm}. Recall that $\bm{H}$ represents the QoI realization matrix for the high-fidelity model, and $\bm{L}$ the corresponding matrix for the low-fidelity model. Our analysis assumes that these matrices are related through matrices $\bm T$ and $\bm E$ according to (\ref{eqn:core_assumption}). Following the discussions of Section \ref{sec:theorem} and, in particular, the bound in (\ref{eq:the_one_bound_pre}), we seek to identify conditions implying that $\Vert\bm T\Vert$ is bounded and $\Vert\bm E\Vert$ is small. 

For $\tau\ge 0$, define 
\begin{align}
\label{eq:t_eps_def}
\epsilon(\tau) \coloneqq \mathop{\arg\min}_{\epsilon}\left\{ \forall \bm{x}\in\mathbb{R}^N:  \quad \|\bm{Hx}\|^2\le \tau\|\bm{Lx}\|^2 + \epsilon\|\bm{x}\|^2 \right\}.
\end{align}
Note that $\epsilon(\tau)$ is well-defined, is a non-increasing function of $\tau$ and, satisfies $\epsilon(\tau) \le \|\bm{H}\|^2$. Defining $\epsilon$ and $\tau$ this way is equivalent to the definition (\ref{eqn:eps_tau_def}) used in Theorem~\ref{thm:algorithm}. We present this equivalence here as a lemma.
\begin{lem}
\label{lem:algorithm}
For any $\tau$, for $\epsilon(\tau)$ defined as in (\ref{eq:t_eps_def}), it follows that
\begin{align}
\label{eqn:epsilon_t}\epsilon(\tau) &= \lambda_{\max}(\bm{H}^T\bm{H}-\tau \bm{L}^T\bm{L}),
\end{align}
i.e., $\epsilon(\tau)$ is the smallest $\epsilon$ such that $\tau \bm{L}^T\bm{L} + \epsilon \bm{I} - \bm{H}^T\bm{H}$ is a semi-positive-definite (SPD) matrix.
\end{lem}
\begin{proof}
For a given $\tau$, and $\epsilon(\tau)$ as from \eqref{eq:t_eps_def}, it follows that for all nontrivial $\bm{x}$,
\begin{align*}
\|\bm{Hx}\|^2&\le \tau\|\bm{Lx}\|^2 + \epsilon(\tau)\|\bm{x}\|^2;\\
\frac{\tau\|\bm{Lx}\|^2 + \epsilon(\tau)\|\bm{x}\|^2 - \|\bm{Hx}\|^2}{\|\bm{x}\|^2} &\ge 0.
\end{align*}
By the definition in \eqref{eq:t_eps_def}, $\epsilon(\tau)$ is the smallest such number making this inequality hold for all $\bm{x}$, meaning that the inequality is equality for some $\bm{x}$, implying
\begin{align*}
\lambda_{\min}(\tau \bm{L}^T\bm{L} + \epsilon(\tau) \bm{I} - \bm{H}^T\bm{H}) &= 0;\\
\lambda_{\min}(\tau \bm{L}^T\bm{L}- \bm{H}^T\bm{H}) &= -\epsilon(\tau);\\
\lambda_{\max}(\bm{H}^T\bm{H} - \tau \bm{L}^T\bm{L}) &= \epsilon(\tau),
\end{align*}
completing the proof.
\end{proof}

We now identify the specific matrices $\bm{T}$ which we consider for our analysis of \eqref{eqn:core_assumption}.
\begin{lem}
\label{lem:one_space}
Let the SVD of $\bm{L}$ be given by
\begin{align*}
\bm{L} = \bm{U}\bm{\Sigma} \bm{V}^T.
\end{align*}
Let $\mc{V}_k$ be the linear subspace spanned by the first $k$ singular vectors in $\bm V$, i.e., the right singular vectors associated with the $k$ largest singular values, and let $\bm{V}_k$ be the matrix of those singular vectors. Let
\begin{align}
\nonumber \bm{P}_{\mc{V}_k} = \bm{V}_k\bm{V}_k^T;\\
\label{eq:tr_def}
\bm{T} \coloneqq \bm{HP}_{\mc{V}_k}\bm{L}^+,
\end{align}
where $\bm{P}_{\mc{S}}$ denotes the orthogonal projection matrix onto $\mc{S}$, and $^{+}$ denotes the Moore-Penrose pseudo-inverse. Then for $k\le \mbox{rank}(\bm{L})$,
\begin{align}
\label{eq:a-tb_def}
 \bm{H}-\bm{T}\bm{L}&= \bm{HP}_{\mc{V}_k^{\perp}}.
\end{align}
\end{lem}
\begin{proof}
  Let $\mathcal{N}(\bs{L})$ denote the nullspace of $\bs{L}$.  Note that
\begin{align*}
\bm{L}^+\bm{L} = \bm{P}_{\mc{N}(\bm{L})^{\perp}}.
\end{align*}
Since $k\le\mbox{rank}(\bm{L})$, then $\mc{V}_k \subset \mc{N}(\bm{L})^{\perp}$. Then from \eqref{eq:tr_def}, it follows that
\begin{align*}
\bm{H}-\bm{T}\bm{L}&= \bm{H} -\bm{H}\bm{P}_{\mc{V}_k}\bm{L}^+\bm{L},\\
& = \bm{H}(\bm{I}-\bm{P}_{\mc{V}_k}\bm{P}_{\mc{N}(\bm{L})^{\perp}})\\
& = \bm{H}(\bm{I}-\bm{P}_{\mc{V}_k}),\\
&= \bm{HP}_{\mc{V}_k^{\perp}}.
\end{align*}
\end{proof}
This result allows us to bound the two key quantities, $\|\bm{E}\| = \|\bm{H}-\bm{T}\bm{L}\|$ and $\|\bm{T}\|$ in terms of $\epsilon(\tau)$.
\begin{lem}
\label{lem:</3</3</3}
Let $\tau$ and $\epsilon(\tau)$ satisfy \eqref{eq:t_eps_def}, and let $\bs{T}$ be as defined in \eqref{eq:tr_def}. Then, for $\sigma_{k}$, the $k$th largest singular value of $\bm{L}$,
\begin{align}
\label{eq:m3a-trb}
\|\bs{E} \|^2 = \|\bm{H}-\bm{T}\bm{L}\|^2 &\le \tau\sigma^2_{k+1} + \epsilon(\tau), \\
\label{eq:m3nt}\|\bm{T}\|^2 &\le \tau +\epsilon(\tau)\sigma_k^{-2}.
\end{align}
%%
%while
%%
%\begin{align}
%\label{eq:m3nt}\|\bm{T}\|^2 &\le \tau +\epsilon(\tau)\sigma_k^{-2}.
%\end{align}
%%
\end{lem}
\begin{proof}
For any $\bm{x}\in\mathbb{R}^N$ with $\|\bm{x}\| \le 1$,
\begin{align*}
\|(\bm{H}-\bm{T}\bm{L})\bm{x}\|^2 &= \|\bm{H}\bm{P}_{\mc{V}_r^\perp}\bm{x}\|^2,\\
& \le \tau\|\bm{LP}_{\mathcal{V}_k^{\perp}}\bm{x}\|^2 + \epsilon(\tau)\|\bm{P}_{\mc{V}_k^\perp}\bm{x}\|^2,\\
& \le \tau\sigma^2_{k+1} + \epsilon(\tau),
\end{align*}
which proves \eqref{eq:m3a-trb}. We have constructed $\bm{T}$ so that
\begin{align}
\label{eq:ufoundasecret}\bm{T} \bm{y} &= \bm{0}, &\bm{y}\in\mc{R}(\bm{L})^{\perp};\\
\label{eq:ucleverperson}\bm{T} \bm{Ly} &= \bm{0}, &\bm{y}\in\mc{V}_k^{\perp},
\end{align}
where $\mc{R}(\bm{L})$ is the range of $\bm L$. It follows that
\begin{align*}
\|\bm{T}\|^2 &= \mathop{\max}\limits_{\|\bm{x}\|\le 1}\|\bm{T}\bm{x}\|^2,\\
&\leftstackrel{\eqref{eq:ufoundasecret}}{=} \mathop{\max}\limits_{\|\bm{Ly}\|\le 1}\|\bm{T}\bm{Ly}\|^2,\\
&\leftstackrel{\eqref{eq:ucleverperson}}{=} \mathop{\max}\limits_{\|\bm{Ly}\|\le 1, \bm{y}\in\mc{V}_k}\|\bm{Hy}\|^2,\\
&\le \tau\|\bm{Ly}\|^2 + \epsilon(\tau)\|\bm{y}\|^2,
\end{align*}
% %
for some $\bs{y}$ achieving the maximum, and the last inequality uses the definition of $\epsilon(\tau)$. Note that $\|\bm{Ly}\| \le 1$, and for $\bm{y}\in\mc{V}_k$, $\|\bm{y}\| \le \sigma_k^{-1}\|\bm{Ly}\|\le \sigma_k^{-1}$. Therefore,
\begin{align*}
\|\bm{T}\|^2 &\le \tau + \epsilon\sigma_k^{-2},
\end{align*}
which proves \eqref{eq:m3nt}.
\end{proof}

The proof of Theorem \ref{thm:algorithm} follows from the combination of the bound provided in (\ref{eq:the_one_bound_pre}) and the results of Lemma~\ref{lem:</3</3</3}. To explain the minimization of (\ref{eqn:main_bound}) over $k$, we note that the matrix $\bm T$ in (\ref{lem:one_space}) can be constructed for any rank $k$ not larger than the rank of the low-fidelity data matrix $\bm L$.

As a brief remark we note that (\ref{eqn:main_bound}) may instead be taken as
\begin{align}
\label{eq:m3err2}
\resizebox{0.9\textwidth}{!}{
$\|\bm{H}-\hat{\bm{H}}\| \le \mathop{\min}\limits_{\tau_1,\tau_2,k\le \mbox{rank}(\bm{L})}\left(1 + \|\bm{C}_L\|\right)\sqrt{\tau_1\sigma^2_{k+1} + \epsilon(\tau_1)} + \|\bm{L}-\hat{\bm{L}}\|\sqrt{\tau_2 + \epsilon(\tau_2)\sigma_k^{-2}},$
}
\end{align}
where the choices $(\tau_1,\epsilon_1(\tau_1))$ and $(\tau_2,\epsilon_2(\tau_2))$ both satisfy \eqref{eq:t_eps_def}. Choosing different values may produce more effective values depending on $\sigma_k$, $\sigma_{k+1}$, $\|\bm{C}_L\|$, and $\|\bm{L}-\hat{\bm{L}}\|$. For simplicity, we have restricted the results of this work to the optimization over a single point $(\tau, \epsilon (\tau))$.

\section*{References}

\bibliographystyle{elsarticle-num}
\bibliography{citations}
\end{document}